\documentclass[12pt,reqno]{article}
mm \textheight=8.9in

\usepackage{amssymb}
\usepackage{amsmath}
\usepackage{amsthm}
\usepackage{pb-diagram}
\usepackage{color}

\newcommand{\lb}{\boldsymbol{l}}

\newcommand{\sms}{\hspace{.7pt}}

\newcommand{\br}[3]{{$#1$}$\lower4pt\hbox{$\tp\atop\raise4pt \hbox{$\scriptscriptstyle{#2}$}$} ${$#3$}}
\newcommand{\tw}[3]{{$#1$}${\,\scriptscriptstyle {#2}}\atop\raise9pt\hbox{$\scriptstyle\tp$} ${$#3$}}
\newcommand{\ttps}[2]{{#1}\raise5pt\hbox{$\lower12pt\hbox{$\scriptstyle\tp$}\atop \lower0pt\hbox{$\tilde\;$}$}\raise4.5pt\hbox{${\scriptstyle{#2}}$}}
\newcommand{\st}[1]{\mbox{${\,\scriptscriptstyle {#1}}\atop\raise5.5pt\hbox{$*$}$}}

\newcommand{\rd}[1]{\mbox{${\,\scriptscriptstyle {#1}}\atop\raise5.5pt\hbox{$\bullet$}$}}
\newcommand{\rt}[1]{\otimes_\chi}
\newcommand{\lt}[1]{\mbox{${\,\scriptscriptstyle {#1}}\atop\raise5.5pt\hbox{$\ltimes$}$}}
\newcommand{\btr}{\raise1.2pt\hbox{$\scriptstyle\blacktriangleright$}\hspace{2pt}}
\newcommand{\btl}{\raise1.2pt\hbox{$\scriptstyle\blacktriangleleft$}\hspace{2pt}}

\newcommand{\lcr}{\raise1.0pt \hbox{${\scriptstyle\rightharpoonup}$}}
\newcommand{\rcr}{\raise1.0pt \hbox{${\scriptstyle\leftharpoonup}$}}

\newcommand{\ttp}{{\lower12pt\hbox{$\tp$}\atop \hbox{$\tilde\;$}}}

\newcommand{\id}{\mathrm{id}}

\newcommand{\Tg}{\mathfrak{T}}

\newcommand{\C}{\mathbb{C}}
\newcommand{\Z}{\mathbb{Z}}

\newcommand{\N}{\mathbb{N}}

\newcommand{\tp}{\otimes}

\newcommand{\U}{U}

\newcommand{\dt}{\delta}

\newcommand{\la}{\lambda}

\newcommand{\End}{\mathrm{End}}

\newcommand{\Span}{\mathrm{Span}}

\newcommand{\Rm}{\mathrm{R}}

\newcommand{\g}{\mathfrak{g}}
\renewcommand{\b}{\mathfrak{b}}

\newcommand{\h}{\mathfrak{h}}

\newcommand{\sib}{\boldsymbol{\si}}
\newcommand{\kb}{\boldsymbol{k}}

\newcommand{\s}{\mathfrak{s}}

\newcommand{\n}{\mathfrak{n}}

\newcommand{\nn}{\nonumber}

\renewcommand{\l}{\mathfrak{l}}

\newcommand{\si}{\sigma}
\newcommand{\al}{\alpha}

\newcommand{\bt}{\beta}

\newcommand{\be}{\begin{eqnarray}}
\newcommand{\ee}{\end{eqnarray}}

\newtheorem{thm}{Theorem}[section]
\newtheorem{propn}[thm]{Proposition}
\newtheorem{lemma}[thm]{Lemma}
\newtheorem{corollary}[thm]{Corollary}

\newtheorem{example}[thm]{Example}

\newcount\prg

\newcommand{\parag}{\advance\prg by1 {\noindent\bf\thesection.\the\prg\hspace{6pt}}}

\begin{document}
\title{Orthogonal basis for the Shapovalov form on $U_q\bigl(\s\l(n+1)\bigr)$}
\author{
Andrey Mudrov\footnote{Partly supported by the RFBR grant 12-01-00207-a.} \vspace{20pt}\\
\small Department of Mathematics,\\ \small University of Leicester, \\
\small University Road,
LE1 7RH Leicester, UK\\
\small e-mail: am405@le.ac.uk\\
}

\date{}
\maketitle

\begin{abstract}
Let $U$ be either classical or quantized universal enveloping algebra of the Lie algebra $\s\l(n+1)$ extended over
the field of fractions of the Cartan subalgebra. We suggest a PBW basis
in $U$ over the extended Cartan subalgebra diagonalizing the
contravariant Shapovalov form on  generic Verma module.
The matrix coefficients of the form are calculated and the inverse of the form is explicitly constructed.
\end{abstract}


\section{Introduction}
The contravariant bilinear form on Verma modules is a fundamental object in the representation theory
of simple complex Lie algebras and quantum groups, which is responsible for many important properties including irreducibility,
\cite{Jan2}.
Its inverse is closely related with intertwining operators \cite{EK}, the dynamical Yang-Baxter equation \cite{ES}, and invariant star product on homogeneous spaces, \cite{AL,KST}.

Contravariant forms on highest weight modules descend from a bilinear form on the
universal enveloping algebra  with values in the Cartan subalgebra. It was
introduced and studied by Shapovalov  \cite{Sh},  who computed the determinant for its restriction to every weight subspace.
It was extended to quantum groups in \cite{DCK}.
The determinant formula was further generalized for parabolic Verma modules over the classical universal enveloping algebras in \cite{Jan1}. These results provided a criterion for the corresponding
 modules to be irreducible, since the  kernel of a contravariant form is invariant.

Applications to mathematical physics require the knowledge of the inverse Shapovalov form, which explicit expression
is an open problem for general simple Lie algebras. The most important advance in this direction was made in \cite{BHST},
where matrix coefficients of the pairing on Mickelsson algebras were calculated. However, \cite{BHST} does not address the Verma modules focusing on different problem. Although the inverse Shapovalov form for the $A_n$ series
can be derived from \cite{BHST}, a self-contained presentation is still missing in the literature. In the present paper we give an independent elementary derivation based on the definition of the quantum group.
We construct the orthogonal basis of the Shapovalov form on $U_q\bigl(\g\l(n+1)\bigr)$  and obtain a similar result for $U\bigl(\g\l(n+1)\bigr)$ via the classical limit. Of course, the classical case can be done directly, in an even
simpler way. The ground field is fixed to $\C$ but can be changed to an arbitrary field of zero characteristic.

We consider a system of "dynamical root vectors" $\hat e_{\pm \mu}$ in the Borel subalgebras.
Upon appropriate ordering, it gives rise to a Poincar\'e-Birkhoff-Witt (PBW) basis over the (extended) Cartan subalgebra.
The vectors $\hat e_{\pm \mu}$ are constructed from the Chevalley generators through generalized commutators
with coefficients in the Cartan subalgebra.
The positive and negative dynamical root vectors are related via
$\omega(\hat e_{\pm\mu})=\hat e_{\mp\mu}$, where $\omega$ is the anti-algebra Chevalley involution.
This PBW system diagonalizes the Shapovalov form on every
Verma module $M_\la$ and is complete if the  highest weight $\la$ is away from a family of hyperplanes. This family is wider that the zero set of the Shapovalov determinant, which is known to be  $\cup_{\al\in R^+}\{\la| (\la+\rho,\al)\in \N\}$
for $U(\g)$ and $\cup_{\al\in R^+}\{\la| q^{2(\la+\rho,\al)}\in q^{2\N}\}$ for $U_q(\g)$. Our set of singular points is
still contained in $\cup_{\al\in R^+}\{\la| (\la,\al)\in \Z\}$ for $U(\g)$ and  in
$\cup_{\al\in R^+}\{\la| q^{2(\la,\al)}\in q^{2\Z}\}$ for $U_q(\g)$. Away from this set, the dynamical PBW system
is a basis.
We compute the matrix coefficients and construct the inverse form for generic weight, off the union
of hyperplanes where some of the matrix coefficients vanish.

The dynamical root vectors project to generators of the Mickelsson algebras associated with a chain
of subalgebras $\s\l(i)\subset \s\l(i+1)$, $i=2,\ldots, n$. Essentially they are raising and lowering
operators participating in  construction of the Gelfand-Zetlin basis in finite dimensional $U_q(\g)$-modules,
\cite{Mol}. Elements of the Gelfand-Zetlin basis are  formed by common eigenvectors of the commutative subalgebra generated by $U_q(\h)$ and the center of $U_q\bigl(\s\l(i)\bigr)$, $i=2,\ldots,n+1$. The dynamical PBW monomials feature the same property and
 form the Gelfand-Zetlin basis in Verma modules.

The paper is organized as follows.
After the preliminary section containing the basics  on the quantum group $U_q(\s\l(n+1))$,
we introduce the dynamical root vectors and study their key properties. Then we show that, upon
an appropriate ordering, the systems of positive and negative dynamical  PBW monomials
give rise to dual bases in right lower  and left
upper Verma modules with respect to the cyclic Shapovalov pairing. We compute the matrix coefficients and construct
the inverse of the cyclic form. Further we pass from the cyclic form to contravariant and prove that the PBW system of negative dynamical root
vectors yields an orthogonal basis. This should be regarded as a refinement of the cyclic result and it is based on a "row-wise commutativity" of dynamical root vectors proved therein.
Further we illustrate the key steps on the example of $A_2$. In the last section, we apply the dynamical root vectors to construction of singular vectors in the Verma modules.

\section{Preliminaries: the quantum group $U_q(\s\l(n+1))$}
For a guide in quantum groups, the reader is referred to \cite{Jan2} or \cite{ChP}, or to the original
paper \cite{D}. In this section we collect the facts about quantum $\s\l(n+1)$ that are
relevant to this exposition.

Let us fix some general notation. We work over the ground field $\C$ of complex numbers.
By $\Z$ we denote the set of all integers, by $\Z_+$ the subset of non-negative
and by $\N$ the subset of strictly positive integers. Given $a,b\in \Z$ we understand by $[a,b]\subset \Z$
the interval of all integers from $a$ to $b$ inclusive. We also use the notation $(a,b]$, $[a,b)$,
and $(a,b)$ for  intervals without one or two boundaries.

Throughout the paper, $\g$ stands for the Lie algebra $\g=\s\l(n+1)$, $n\geqslant 1$. The case $n=1$ is
trivial, and we are mostly interested in $n\geqslant 2$. Fix a Cartan subalgebra  $\h\subset \g$
and let $R\subset \h^*$ denote the root system of $\g$ with a subsystem $R^+$ of positive roots, relative to $\h$.
The choice of $R^+$ facilitates  a triangular decomposition,
$\g=\n^-\oplus \h\oplus \n^+,$
where $\n^\pm$ are nilpotent Lie subalgebras corresponding to the positive and negative roots.
Let $(.,.)$ designate the canonical inner product on $\h^*$.

Denote  by $\Pi^+\subset R^+$ the basis of simple positive roots $\{\al_1,\ldots, \al_n\}$, with the standard ordering
determined up to the inversion by the condition $(\al_i,\al_j)=0$ for $|i-j|>1$.
For any pair of integers $i,j\in [1,n]$ such that $i\leqslant j$
let  $\g_{ij}\subset \g$ be the Lie subalgebra
$\s\l(j-i+2)$ corresponding to the roots $\al_i,\ldots, \al_j \in \Pi^+$.
We also consider the Cartan subalgebra $\h_{ij}=\g_{ij}\cap \h$ and nilpotent subalgebras $\n_{ij}^\pm =\g_{ij}\cap \n^\pm$,
so that
$\g_{ij}=\n^-_{ij}\oplus \h_{ij}\oplus \n^+_{ij}$
is a triangular decomposition compatible with the decomposition of $\g$.

We assume that $q\in \C$ is not a root of unity and define $[x]_q=\frac{q^x-q^{-x}}{q-q^{-1}}$   for  an indeterminate $x$.
The quantum group $U_q(\g)$ is a $\C$-algebra generated by $e_i,f_i,t^{\pm 1}_i$,  $i\in[1,n]$, subject to the
Chevalley relations
$$
t_ie_j=q^{(\al_i,\al_j)}e_jt_i,\quad t_if_j=q^{-(\al_i,\al_j)}f_jt_i,\quad [e_i,f_j]=\dt_{ij}\frac{t_i-t^{-1}_i}{q-q^{-1}},
$$
and the  Serre relations
$$
e_i^2e_j-[2]_qe_ie_j e_i+e_je_i^2=0, \quad f_i^2f_j- [2]_qf_if_j f_i+f_jf_i^2=0, \quad |i-j|=1,
$$
$$
[e_i,e_j]=0, \quad [f_i,f_j]=0, \quad |i-j|>1.
$$
The elements $e_i$ and $f_i$ are called, respectively, the positive and negative Chevalley generators.
The assignment $\omega\colon t_i\mapsto t_i$, $\omega\colon e_i\mapsto f_i$, $\omega\colon f_i\mapsto e_i$
extends to an anti-algebra automorphism of $U_q(\g)$ called Chevalley involution.

The quantum group can be also defined as an algebra over  the
ring of fractions of $\C[q,q^{-1}]$ over the multiplicative system generated by $q^m-1$, $m\in \N$.
Its $\C[\![\hbar]\!]$-version is a $\C[\![\hbar]\!]$-extension of $U_q(\g)$ completed in
the $\hbar$-adic topology. The extension goes through the embedding $\C[q,q^{-1}]\to \C[\![\hbar]\!]$, $q\mapsto e^{\hbar}$.
The corresponding relations translate to
$$
[h_i,e_j]=(\al_i,\al_j)e_j,\quad [h_i,f_j]=-(\al_i,\al_j)f_j,\quad [e_i,f_j]=\dt_{ij}[h_i]_q.
$$
upon the substitution $t^{\pm1}_i=q^{\pm h_i}$. This algebra, denoted by $U_\hbar(\g)$, is a deformation
of the classical universal enveloping algebra $U(\g)$.
It is still convenient to use the notation $[h_i]_q=\frac{t_i-t^{-1}_i}{q-q^{-1}}$ and $q^{h_i}=t_i$ when working with
$U_q(\g)$. This makes sense of $[h]_q\in U_q(\g)$ for any linear combination $h=c_0+\sum_{i=1}^nc_i h_i$
with integer $c_i$, $i>0$, and arbitrary complex $c_0$.
We denote by $U_q(\h)$ the subalgebra in $U_q(\g)$ generated by $\{t_i^{\pm 1}\}_{i=1}^n$. This $q$-version of the Cartan subalgebra
is the polynomial ring on a torus, while $U(\h)$ is a polynomial ring on a vector space.
Note that $\h\not \subset U_q(\h)$ contrary to $U_\hbar(\h)$, which stands for the subalgebra in
$U_\hbar(\g)$ generated by $\{h_i\}_{i=1}^n$.

Observe that $U_q(\g_{ij})$ is a natural subalgebra in $U_q(\g)$ for any pair $i,j\in[1,n]$ such that $i\leqslant j$.
Here are other subalgebras of importance in $U_\hbar(\g)$.
The elements $e_i$ and $f_i$  generate, respectively, the subalgebras  $U_q(\n^+)$ and $U_q(\n^-)$.
Their $\C[\![\hbar]\!]$-extensions $U_\hbar(\n^\pm)$ are deformations of the classical
universal enveloping algebras $U(\n^\pm)$.
The quantum Borel subalgebras $U_q(\b^\pm)$ are generated by $U_q(\n^\pm)$ over $U_q(\h)$.

All positive roots in $R^+$ are sums $\al_i+\ldots+\al_j$, where $i\leqslant j$.
Put $e_{ii}=e_i$, $f_{ii}=f_i$ and extend this definition inductively by
\be
e_{ij}:=[e_{i+1\sms j},e_i]_q,&&
f_{ij}:=[f_i,f_{i+1\sms j}]_q\nn
\ee
for $i< j$. Here $[x,y]_q$ is the generalized commutator $xy-qyx$.
Along with $e_{ii}$, $f_{ii}$ we will also use the usual notation $e_i$, $f_i$.
Note that the positive and negative root vectors are related via the Chevalley involution,
$
\omega(f_{ij})=e_{ij}
$.

We define $\n^\pm$ in the $q$-case as the linear spans
  $\n^+=\{e_{km}\}_{k\leqslant m}\subset U_q(\g)$ and  $\n^-= \{f_{km}\}_{k\leqslant m}\subset U_q(\g)$.
These are  $U_\hbar(\h)$-submodules, which are trivial deformations of the classical $U(\h)$-modules $\n^\pm\subset \g$.
Similarly, we put $\n_{km}^\pm= \n^\pm \cap U_q (\g_{km})$, so that
$\n_{km}^+=\Span\{e_{ij}\}_{k\leqslant i\leqslant j\leqslant m}$ and $\n_{km}^-=\Span\{f_{ij}\}_{k\leqslant i\leqslant j\leqslant m}$.

\begin{lemma}
\label{basics comms}
Suppose that $k\in (i,j)\subset [1,n]$. Then $[f_k,f_{ij}]=0=[e_k,f_{ij}]$.
Further, $[f_{ij},f_i]_{q}=0$,  $[f_j,f_{ij}]_{q}=0$,
$[e_i,f_{ij}]=f_{i+1\sms j}q^{-h_i}$, $[e_j,f_{ij}]=-qf_{i\sms j-1}q^{h_i}$.
\end{lemma}
\begin{proof}
Direct calculation.
\end{proof}
In what follows, we deal with a general algebraic concept, which we recall here.
Consider a unital associative algebra  $A$ and a non-empty subset $I\subset A$. Let $AI$ denote the left ideal generated
by $I$. We denote by $A^I$ the subset of elements $x\in A$ such that
$Ix\subset AI$. We write simply $A^a$ when $I=\{a\}$ consists of one element $a$. Obviously $A^I$ is not empty,
$A^I\supset AI$, and is
a subalgebra in $A$. It is the normalizer of $AI$, i.e. the maximal subalgebra in $A$ where $AI$ is a two-sided
ideal.  For every $x \in A^I$ the map $x\colon AI \to AIx\subset AI$ amounts to  an
anti-homomorphism $A^I\mapsto \End_A(AI)$, where
the ideal $AI$ is  regarded as a  natural submodule of the regular left $A$-module.
Obviously $x \in A^I$ if and only if $[x,I]\subset AI$.

Similarly one defines the normalizer ${}^I\!\!A$ of the right ideal $IA$ generated by $I$.
As in the left case, $x \in {}^I\!\!A$ if and only if $[x,I]\subset IA$.

In our setting, $A$ will be $U:=U_q(\g)$. If $I$ a subset of simple positive root vectors and $\g'$ is
the corresponding reductive  subalgebra in $\g$, the quotient $A^I/AI$ is the Mickelsson algebra $S(\g,\g')$,
\cite{Mick}.

We finish our introduction to the quantum special linear group with two lemmas that will be used in what follows.
Let $S_m$ denote the symmetric group of permutations of $m$ symbols.
\begin{lemma}
\label{nilp}
Suppose that $m\in [2,n]$. For any $\si\in S_m$, the Chevalley monomial $f_{\si(1)}\ldots f_{\si(m)}$ belongs to $\>\>{}\!\!^{\n^-_{2\sms m}}U$.
Moreover, $f_{\si(1)}\ldots f_{\si(m)}\in {\n^-_{2\sms m}}U$, provided  $\si\not=\id$.
\end{lemma}
\begin{proof}
Consider the case $\si=\id$ first, using induction on $m$.
For $m=2$ the statement immediately follows from the Serre relations:
$(f_1f_2) f_2=f_2([2]_qf_1f_2-f_2f_1)\in {\n^-_{2\sms m}}U$, while the second statement is obvious.
Suppose that $m>2$ and the lemma has been proved for all $i$ from the interval $[2,m)$.
Then, for such $i$, the Serre relations give
$$
f_1\ldots f_{m} f_i=f_1\ldots f_{i}f_{i+1}f_i\psi=
\frac{1}{[2]_q}f_1\ldots f_{i-1}f_{i}^2f_{i+1}\psi
+\frac{1}{[2]_q}f_1\ldots f_{i-1}f_{i+1}f_i^2\psi,
$$
where $\psi=f_{i+2}\ldots f_{m}$ and $\psi=1$ if $i=m-1$.
By the induction assumption, the first term belongs to $ \n^-_{2\sms i}U$.
In the second term, $f_{i+1}$ commutes with $f_1\ldots f_{i-1}$. Therefore, the second term belongs to $f_{i+1}U\subset \n^-_{2\sms m}U$, and the sum lies in $ \n^-_{2\sms m}U$.
For $i=m$, we have
$$
f_1\ldots f_{m}f_{m}=[2]_qf_1\ldots f_{m-2}f_{m}f_{m-1}f_{m}-f_1\ldots f_{m-2}f_{m}^2f_{m-1}
\in f_{m}U\subset \n^-_{2\sms m}U.
$$
This proves the statement for all $m$ and  $\si=\id$.

Suppose that $\si\not= \id$.
The statement is obvious if $\si(1)\not =1$. Otherwise let $i\in[2,m)$ be the least integer such that $\si(i)\not= i$. Put $\psi=f_{\si(i+1)}\ldots f_{\si(m)}$.
Then
$$f_{\si(1)}\ldots f_{\si(m)}=f_{1}\ldots f_{i-1}f_{\si(i)}\psi=f_{\si(i)}f_{1}\ldots f_{i-1}\psi
\in \n^-_{2\sms m}U,
$$
as $\si(i)>i\geqslant 2$. This proves the statement  for $\si\not= \id$.
\end{proof}
\begin{lemma}
\label{e1kf1m}
Suppose that integers $i,j,k,m\in [1,n]$ satisfy the inequalities  $i\leqslant j\leqslant k<m$. Then for all $u\in U_q(\n^+_{i\sms k})$, $[u,f_{j\sms m}]\in \n^-_{j+1\sms m}U$.
\end{lemma}
\begin{proof}
Introduce a grading in
$U_q(\n^+)$ setting $\deg e_j=1$ for all $j\in [1,n]$. Let $u\in U_q(\n^+_{i\sms k})$ be a Chevalley monomial.
The statement is trivial for zero degree $u$, so we assume $\deg u>0$.
Present $u$ as a product $u=u'e_l$ for some $e_l, u'\in  U_q(\n^+_{i\sms k})$.
If $\deg u'=0$ and $u=e_l$, then the statement follows from the formula
$[e_l,f_{j\sms m}]=\dt_{jl}f_{j+1\sms m}q^{-h_j}$, cf. Lemma \ref{basics comms}.
For $\deg u\geqslant 1$, induction on $\deg u$ gives
$$
u f_{j\sms m}=u' f_{j m} e_l+ \dt_{jl} u'f_{j+1\sms m}q^{-h_j}
\in
f_{j m} u' e_l+ \n^-_{j+1\sms m}U+ \dt_{jl} f_{j+1\sms m} u'q^{-h_j}+ \n^-_{j+2\sms m}U,
$$
where the last summand is present only if $j+2\leqslant m$. The right-hand side is contained in
$f_{j m} u+ \n^-_{j+1\sms m}U$, as required.
\end{proof}

\section{Dynamical root vectors}
\label{DRV}
We set up an ordering on  positive  root vectors $e_{ij}$ induced by the lexicographic ordering on
 pairs $(i,j)$, $i\leqslant j$. The negative root vectors $f_{ij}$ are ordered in the opposite way.
These orderings are normal and compatible with a reduced decomposition
of the maximal element in the Weyl group of $\g$. The ordered systems of root vectors generate a PBW basis
in the algebras $U_q(\n^\pm)$, \cite{ChP}.
The Shapovalov form, which is the subject of our interest, is very complicated in
this basis. We need a new basis suitable for our study, possibly on the extension of $U_q(\g)$ over
the ring of factions of $U_q(\h)$ over some multiplicative system. This basis is introduced in this section.

Put $h_{ik}:=h_i+\ldots +h_k+k-i$ for all positive integer $i,k$ such that $i\leqslant k$. The difference $k-i$
is equal to $(\rho,\al_{ik})-1$, where $\al_{ik}=\al_i+\ldots+\al_k\in R^+$.
We define dynamical root vectors $\hat f_{ik}\in U_q(\b^-)$ and $\hat e_{ik}\in U_q(\b^+)$  for all pairs $i,k\in[1,n]$ of integers
such that $i\leqslant k$.
For $k=i\in [1,n]$ we put $\hat e_{ii}=e_i$ and $\hat f_{ii}=f_i$. For  $i< k$
we proceed recursively by
$$
\hat e_{ik}=q^{-1}[h_{i+1\sms k}]_q[\hat e_{i+1\sms  k},e_i]_{q}+q^{h_{i+1\sms  k}}\hat e_{i+1\sms k}e_i,
\quad
\hat f_{ik}=q^{-1}[f_i,\hat f_{i+1\sms k}]_{q}[h_{i+1\sms k}]_q+f_i\hat f_{i+1\sms k}q^{h_{i+1\sms k}},
$$
The right-hand side can be expressed through "generalized commutators" with coefficients from the Cartan subalgebra.
For instance,
$$
\hat e_{ik}=[h_{i+1\sms k}+1]_q\hat e_{i+1\sms k}e_{i}-[h_{i+1\sms k}]_qe_{i}\hat e_{i+1\sms k},
\quad
\hat f_{ik}=f_{i}\hat f_{i+1\sms k}[h_{i+1\sms k}+1]_q-\hat f_{i+1\sms k}f_{i}[h_{i+1\sms k}]_q.
$$
Note that the Cartan coefficients in $\hat e_{i+1\sms k}$ commute with $e_{i}$ and can be gathered on the left.
Similarly to the standard root vectors, $\omega(\hat e_{ij})=\hat f_{ij}$.

The name dynamical follows the analogy with the dynamical Yang-Baxter equation from the mathematical physics literature,
\cite{ES}.
In a representation, the Cartan coefficients are specialized at the weight of a particular vector the elements
$\hat e_{ij}$ and $\hat f_{ij}$ act upon.  This dependence on the weight is "dynamical" rather than "statical" since the Cartan coefficients are not central in $U_q(\g)$.

The key properties of dynamical root vectors are described by the following proposition.
\begin{propn}
\label{key}
For all integer $i,j,k\in [1,n]$ such that $i<j,k$,
$$
\hat f_{ij}\in U^{e_k},\quad [e_i,\hat f_{ij}]=\hat f_{i+1\sms  j}[h_{ij}]_q\mod Ue_i,
$$
$$
\hat e_{ij}\in \>\>{}\!\!^{f_k}U,\quad [\hat e_{ij},f_i]=[h_{ij}]_q\hat e_{i+1\sms j}\mod f_iU.
$$
\end{propn}
\begin{proof}
We will check only the first line. The second line is obtained from it via the Chevalley involution.

It is obvious that $\hat f_{i j}\in U^{e_k}$ for $k>j$, so we assume $i<k\leqslant j$. For $j={i+1}$ we have
$
[e_{j},\hat f_{ij}]=[f_i,[h_{j}]_{q}]_{q}q^{-1}[h_{j}]_q+f_i [h_{j}]_qq^{h_{j}}
$
modulo  $Ue_j$.
The retained terms give
$$
f_i\bigl(([h_j]_{q}-q[h_j+1]_{q})q^{-1}[h_j]_q+q^{h_j}[h_j]_q\bigr)=0,
$$
hence $[e_j,\hat f_{ij}]\in Ue_j$, as required.
For the right equality in the first line, we have
$$
[e_i,\hat f_{ij}]=  [[h_i]_q,f_j]_{q} q^{-1}[h_j]_q+[h_i]_q f_jq^{h_j} +\ldots
=  f_jq^{-h_i} q^{-1}[h_j]_q+ f_j [h_i+1]_{q}q^{h_j}\ldots
$$
where we have omitted the terms from $Ue_i$. Modulo those terms, the last expression  is
equal to $f_j[h_i+h_j+1]_{q}=f_{j}[h_{ij}]_{q}$ for $j=i+1$.
This proves the proposition for $j=k=i+1$.

Further we do induction on $j-i$. The case $j-i=1$ is already done. Suppose
that the proposition is proved for $j-i$ up to $l-1>0$. Then
$[e_k,\hat f_{i+1\sms j}]\in Ue_k$ for $k\in [i+2,j]$. This immediately implies
the inclusion $[e_k,\hat f_{ij}]\in U^{e_k}$ for such $k$, thanks to the recursive presentation
of $\hat f_{ij}$ through $\hat f_{i+1\sms j}$. For $k=i+1$ we have
$$
[e_{k},\hat f_{ij}]=[f_i,[e_{k},\hat f_{kj}]]_{q}q^{-1}[h_{kj}]_q
+f_i[e_{k},\hat f_{kj}]q^{h_{kj}}+\ldots,
$$
where the omitted terms lie in $Ue_{k}$. By the induction assumption, the remaining terms give
$$
[f_i,\hat f_{k+1\sms j}[h_{kj}]_q]_{q}q^{-1}[h_{kj}]_q
+
f_i\hat f_{k+1\sms j}q^{h_{kj}}[h_{kj}]_q,
$$
up to the terms from $Ue_{i+1}$.
This is equal to the product of $f_i\hat f_{k+1\sms j}$ (observe that $f_i$ commutes with $\hat f_{k+1\sms j}=\hat f_{i+2\sms j}$) and the Cartan factor
$$
[h_{kj}]_q\bigl(([h_{kj}]_q-q^{}[h_{kj}+1]_q\bigr)q^{-1}
+
q^{h_{kj}})=0.
$$
Therefore, $[e_{i+1},\hat f_{ij}]\in Ue_{i+1}$, as required.

To complete the induction, we need to check the rightmost equality:
\be
[e_i,\hat f_{ij}]
&=&
[[h_i]_q,\hat f_{i+1\sms j}]_{q}q^{-1}[h_{i+1\sms j}]_q+[h_i]_q\hat f_{i+1\sms j}q^{h_{i+1\sms j}}+\ldots
\nn\\
&&\hat f_{i+1\sms j}\Bigl(\bigr([h_i+1]_q-q[h_i]_q\bigr)q^{-1}[h_{i+1\sms j}]_q+[h_i+1]_qq^{h_{i+1\sms j}}\Bigr)+\ldots
\ee
where we have dropped the terms from $Ue_i$. The Cartan factor in the brackets  is
$$
q^{-h_i}q^{-1}[h_{i+1\sms j}]_q+[h_i+1]_qq^{h_{i+1\sms j}}=[h_{ij}]_q.
$$
This completes the induction on $l=j-i$ and the proof of the proposition.
\end{proof}

Let $h_\al\in \h$ denote the element determined by $\al(h_\al)=(\la,\al)$ for all $\la\in \h^*$.
Consider the multiplicative system in $U_q(\h)$ generated by $[h_{\al}+m]_q$, $\al\in R^+$, $m\in \Z$,
and denote by $\hat U_q(\h)$ the ring of fractions of $U_q(\h)$ over this system. One can check
that there is a natural extension, $\hat U_q(\g)$, of $U_q(\g)$ over $\hat U_q(\h)$.
The algebra $\hat U_q(\g_{ij})$  contains an idempotent $p_{i\sms j}$ of zero weight  such that
$p_{ij}\hat U_q(\g)=\{x\in \hat U_q(\g)\colon \n_{ij}^+x=0\}$, $\hat U_q(\g)p_{ij}=\{x\in \hat U_q(\g)\colon x\n_{ij}^-=0\}$, \cite{AST,KT}.
It is called  extremal projector of the subalgebra $\hat U_q(\g_{ij})$.

\begin{propn}
The vector $\hat f_{ij}p$ is equal to $p f_{ij}\prod_{l=i+1}^{j}[h_{l\sms j}+1]_q$,
 where $p=p_{i+1\sms j}$.
\end{propn}
\begin{proof}
By construction,
 $\hat f_{ij}p$ belongs to $p U_q(\b^-)$ and hence to $p U_q(\b^-)p$.  By Lemma \ref{nilp},
$p f_{ij}$ belongs to $U_q(\b^-)p$ and hence to $p U_q(\b^-)p$. On the other hand,
$p f_{ij}p=p f_{\al_i}\ldots f_{\al_j}p$ is a unique, up to a scalar factor, vector of weight $\al_{ij}$ in $p U_q(\g_-)p $. Now observe that $f_{\al_i}\ldots f_{\al_j}$ enters $\hat f_{ij}$ with the Cartan coefficient $\prod_{l=i+1}^{j}[h_{l\sms j}+1]_q$.
\end{proof}

It follows that $\hat e_{ij}$ and $\hat f_{ij}$ generate a PBW basis in $\hat U_q(\g)$ over $\hat U_q(\h)$.

\section{Verma modules}
Thanks to a PBW basis, the algebra $U_q(\g)$ is a free
$U_q(\n^-)-U_q(\n^+)$-bimodule  generated by $U_q(\h)$.
The triangular factorization $U_q(\g)=U_q(\n^-)U_q(\h)U_q(\n^+)$ gives rise to the direct sum decomposition
$U_q(\g)=U_q(\h)\oplus [\n^-U_q(\g)+U_q(\g)\n^+]$, which facilitates a projection $\pi\colon  U_q(\g)\to U_q(\h)$.
The Shapovalov form is a linear mapping $\U_q(\g)\tp \U_q(\g)\to U_q(\h)$, defined as  the
composition
$$
\U_q(\g)\tp \U_q(\g)\stackrel{\omega\tp \id}{\longrightarrow} \U_q(\g)\tp\U_q(\g) \longrightarrow \U_q(\g)
\stackrel{\pi}{\longrightarrow} \U_q(\h),
$$
where the middle arrow is the multiplication. The form is $\omega$-contravariant, i.e. the conjugation
operation factors through $\omega$.
The left ideal $\U_q(\g)\n^+$ lies in the kernel of the form, which therefore restricts to
the quotient $\U_q(\g)/\U_q(\g)\n^+$.

It is convenient to drop the extra structure of Chevalley involution and consider
pairings between left and right modules, with  cyclicity in place of contravariance.
Recall that a pairing $\langle .,.\rangle \colon V\tp W$ between a right module $V$ and left module $W$
is called cyclic if $\langle xu,y\rangle =\langle x, uy\rangle $ for all $x\in V$, $u\in W$, and $u\in \U_q(\g)$.
Specifically the cyclic Shapovalov form is defined similarly to contravariant but without the
first arrow. It induces a cyclic pairing between the right and left quotient modules $\n^-\U_q(\g)\backslash \U_q(\g)$ and
$\U_q(\g)/\U_q(\g)\n^+$.

The Shapovalov form on $\U_q(\g)$ is equivalent to a family of forms on Verma modules parameterized by the
highest weight $\la\in \h^*$.
Consider a one dimensional representation of the Cartan subalgebra $U_q(\h)$ determined by the
assignment
$t_i\mapsto q^{\la_i}\in \C$, where $\la_i=(\la,\al_i)$.
It extends to a representation of
$U_q(\b^\pm)$ by letting $\la(\n^\pm)=0$.  We regard
$\C$ as a left $U_q(\b^+)$-module and right $U_q(\b^-)$-module with respect to these extensions
and denote it by $\C_\la$.
Define the right and left Verma  $U_q(\g)$-modules $M_\la^\star$ and $M_\la$ to be the induced modules
$$
M_\la^\star =\C_\la\tp_{U_q(\b^-)}U_q(\g), \quad M_\la=U_q(\g)\tp_{U_q(\b^+)}\C_\la,
$$
When restricted to the Cartan subalgebra, $M_\la^\star$ is isomorphic to $\C_\la\tp U_q(\n^+)$,
while $M_\la$ is isomorphic to $U_q(\n^+)\tp \C_\la$.
Denote by $v_\la^\star\in M_\la^\star$ and $v_\la \in M_\la$ their canonical  generators. They
carry  the  highest weights.

The cyclic Shapovalov pairing $M^\star_\la\tp M_\la \to \C$ is defined by
$$
\langle v^\star_\la x, y v_\la\rangle =\la \bigl(\pi(xy)\bigr), \quad  x, y\in U.
$$
By construction, it is normalized to $\langle v^\star_\la , v_\la\rangle =   1$ and it is a unique cyclic pairing
between $M^\star_\la$ and $M_\la$ that satisfies this condition.
In order to simplify  formulas, we suppress the brackets  and write simply $v^\star_\la x\tp yv_\la\mapsto v^\star_\la xyv_\la$ thanks to the cyclicity. The subspaces  of different weights in $M_\la^\star$ and $M_\la$ are
orthogonal. The module $M_\la$ (equivalently, $M^\star _\la$) is irreducible if and only if this form is non-degenerate.

Recall that a vector in $M_\la$ is called singular if it is annihilated by $\n^+$.
Similarly,  a vector in $M_\la^\star$ is called singular if it is annihilated by $\n^-$.
Singular vectors generate submodules, where they carry the highest weights.
For a subalgebra $\g_{ij}\subset \g$ we say that a vector in $M_\la$ is $\g_{ij}$-singular or $\n^+_{ij}$-singular
if it is killed by $\n^+_{ij}$. Similarly,  we say that a vector in $M_\la^\star$ is $\g_{ij}$-singular or $\n^-_{ij}$-singular
if it is killed by $\n^-_{ij}$

It is also convenient to extend the form to a cyclic paring $M_\mu^\star\tp M_\la\to \C$ by setting it nil for
$\mu\not =\la$.
Given  a root subsystem $\Pi'\subset \Pi$, consider the corresponding semisimple  Lie subalgebra $\g'\subset \g$.
Suppose vectors $v_{\la'}\subset M_\la$ and  $v_{\mu'}^\star\subset M_\la^\star $ are $\g'$-singular and
consider the $U_q(\g')$-submodules $M'_{\la'}\subset M_\la$ and ${M^\star_{\mu'}}'\subset M_\la^\star$
generated by  $v_{\la'}$ and  $v_{\mu'}^\star$.
\begin{propn}
The restriction of the $U_q(\g)$-cyclic form
$M_{\la}^\star\tp  M_\la\to \C$ to ${M^\star_{\mu'}}'\tp M_{\la'}'$
is proportional to the $U_q(\g')$-cyclic form ${M^\star_{\mu'}}'\tp M_{\la'}'\to \C$.
\end{propn}
\begin{proof}
The restriction of the form to ${M^\star_{\mu'}}'\tp M_{\la'}'$ is cyclic with respect to $U_q(\g')$.
A cyclic bilinear form between right and left Verma modules is unique
up to an overall factor.
\end{proof}

\section{Diagonalization of the Shapovalov form}

Let $\Tg=\Z_+^{\frac{n(n+1)}{2}}$ designate
the set of triangular arrays $\lb=(l_{ij})_{1\leqslant i\leqslant j\leqslant n}$ with non-negative integer entries $l_{ij}$.
For every $\lb\in \mathfrak{T}$ and $k\in [1,n]$ we denote by $\lb_k\in \Z_+^{n-k+1}$
its $k$-th row  $(l_{kj})_{k\leqslant j\leqslant n}$.
Define
$$
\begin{array}{cccccc}
f(\lb_k)&=&f_{kn}^{l_{kn}}\ldots f_{kk}^{l_{kk}}\in U(\b^-),
&
f(\lb)&=&f(\lb_n)\ldots f(\lb_1),
\\[3pt]
 e(\lb_k)&=& e_{kk}^{l_{kk}}\ldots  e_{kn}^{l_{kn}}\in U(\b^+),
&
 e(\lb)&=& e(\lb_1)\ldots  e(\lb_n).
\end{array}
$$
The set $\{f(\lb), e(\lb)\}_{\lb\in \Tg}\subset U_q(\g)$ is a PBW  basis over $U_q(\h)$.
Similarly we define $\hat f(\lb)$ and $\hat e(\lb)$ using the dynamical root vectors in place
of standard.
We call $\{\hat f(\lb), \hat e(\lb)\}_{\lb\in \Tg}$ dynamical PBW system. In what follows, we study
the set of vectors
\be
\hat f(\lb)v_\la\in M_\la,
\quad
v_\la^\star \hat e(\lb)\in M_\la^\star, \quad \lb\in \Tg.
\label{cket}
\ee
We prove that, upon a normalization,  they form dual bases in generic $M_\la$ and $M_\la^\star $ with
respect to the cyclic pairing. With respect to the contravariant form on generic $M_\la$,
the system $\{\hat f(\lb)v_\la\}_{\lb\in \Tg}$ is an orthogonal basis.

Note that the ordering of the dynamical root vectors is the same lexicographic
ordering of the standard root vectors set up in Section \ref{DRV}. We call it {\em normal}.
We have to consider different row-wise orderings as well.
Let $\sib=(\si_n,\ldots,\si_1)\in S_n\times\ldots\times  S_1$ be an $n$-tuple of permutations.
Define $\hat e_{\sib}(\lb_k)=\si_k\bigl(\hat e(\lb_k)\bigr)$ to be the result of permutation $\si_k$
applied to the simple factors of $\hat e(\lb_k)$ and put $\hat e_{\sib}(\lb)=\hat e_{\si_1}(\lb_1)\ldots  \hat e_{\si_n}(\lb_n)$.
We prove in Section \ref{SecCShF} that $\hat e_{\sib}(\lb)$ is independent of $\sib$ but  we have to distinguish between different
orderings until then. We will suppress the subscript $\sib$ and understand by
$\hat e(\lb)$ a monomial with arbitrary although fixed ordering. This convention stays in effect until the end of the section.
In the subsequent sections, we use only two orderings: the normal and an alternative, for which we fix a special notation.

The basis of positive (negative) root vectors allows us to identify the factorspaces $\n^\pm_{ij}/\n^\pm_{kj}$ with
the linear complements $\n^\pm_{ij}\ominus \n^\pm_{kj}\subset \n^\pm_{ij}$,
for all $i,j,k\in [1,n]$ such that $i\leqslant k \leqslant j$. By $U_q(\n^\pm_{ij}/\n^\pm_{kj})$
we denote the subalgebras in $U_q(\g)$ generated by $\n^\pm_{ij}/\n^\pm_{kj}$.

Similarly we define $U_q(\h)$-submodules $\hat \n^+_{ij}= \Span \{\hat e_{lk}\}_{i\leqslant l \leqslant k \leqslant j}$,
$\hat \n^-_{ij}= \Span \{\hat f_{lk}\}_{i\leqslant l \leqslant k \leqslant j}$ and
 $\hat \n^\pm_{ij}/\hat\n^\pm_{kj}=\hat \n^\pm_{ij}\ominus \hat\n^\pm_{kj}\subset \hat \n^\pm_{ij}$.
By $U_q(\hat \n^\pm_{ij})\subset U_q(\b^\pm_{ij})$ we denote the subalgebras generated by $\hat \n^\pm_{ij}$
and by $U_q(\hat \n^\pm_{ij}/\hat \n^-_{kj})$ the subalgebras generated
by  $\hat \n^\pm_{ij}/\hat \n^+_{kj}$.
Clearly $U_q(\hat \n^-_{ij})v_\mu \subset U_q(\n^-_{ij})v_\mu $ and
$v_\mu^\star U_q(\hat \n^+_{ij}) \subset v_\mu^\star U_q(\n^+_{ij})$ for all weight vectors $v_\mu$, $v_\mu^\star$.
The monomial structure of $f(\lb)v_\la$  is compatible with the factorization
$$
U_q(\hat \n_{1\sms n}^-)v_{\la}=
U_q(\hat \n_{kn}^-)U_q(\hat \n^-_{1\sms n}/\hat \n_{kn}^-)v_{\la}
=U_q(\hat \n_{nn}^-)U_q(\hat \n_{n-1\sms n}^-/\hat \n_{nn}^-)\ldots U_q(\hat \n_{1\sms n}^-/\hat \n_{2\sms n}^-)v_{\la}.
$$
Similarly, the vector  $v_\la^\star e(\lb)$ is factorized in accordance with
$$
v^\star_\la U_q(\hat \n^+_{1\sms n})=
v^\star_\la U_q(\hat \n^+_{1\sms n}/\hat \n_{kn}^+)U_q(\hat \n_{kn}^+)
=v^\star_\la U_q(\hat \n_{1\sms n}^+/\hat \n_{2\sms n}^+)\ldots U_q(\hat \n_{n-1\sms n}^+/\hat \n_{nn}^+) U_q(\hat \n_{nn}^+).
$$
We shall see in Section \ref{SecCShF} that the algebras $U_q(\hat \n_{i n}^\pm/\hat \n_{i+1\sms n}^\pm)$ are commutative.
\begin{lemma}
\label{sing+-}
Suppose that  $1\leqslant k\leqslant n$. Then  all vectors from
$U_q(\hat \n^-_{1\sms n}/\hat \n_{kn}^-)v_{\la}$ and $v_{\la}^\star U_q(\hat \n^+_{1\sms n}/\hat \n_{kn}^+)$ are $\g_{kn}$-singular.
\end{lemma}
\begin{proof}
An immediate consequence of Proposition \ref{key}.
\end{proof}

Fix $\lb\in \Tg$ and define a sequence of weights $(\la_{\lb,i})_{i=0}^n\subset \h^*$ by
\be
\la_{\lb,0}=\la,\quad \la_{\lb,i}=\la_{\lb,i-1}
-\sum_{i\leqslant j\leqslant k\leqslant n} l_{ik}\al_{j}, \quad i=1,\ldots,n.
\label{partial_weight}
\ee
These are the weights of the vectors $\hat f(\lb_{i})\ldots \hat f(\lb_{1})v_\la$.
Note that the difference $\la_{\lb,i}-\la_{\lb,i-1}$ depends only on $\lb_i$ and not on $\la$.
Define vectors
\be
v^\star_{\la_{\kb,0}}=v^\star_{\la},\quad v_{\la_{\lb,0}}=v_\la
,
\quad
v^\star_{\la_{\kb,i}}=v_\la^\star\hat e(\kb_n)\ldots\hat e(\kb_i)\in
M_\la^\star,
\quad
 v_{\la_{\lb,i}}=\hat f(\lb_{i})\ldots \hat f(\lb_{1})v_\la
M_\la,
\label{partial_sing}
\ee
 $i\in[1,n]$, of weights $\la_{\lb,i}$ (mind the right action of $U_q(\h)$ on $M_\la^\star$).
\begin{propn}
\label{diagonal}
For all $\kb,\lb \in \Tg$, the matrix coefficient
$
v_\la^\star \hat e(\kb)
\hat f(\lb)v_\la
$
is nil unless $\kb=\lb$.
\end{propn}
\begin{proof}
Due to Lemma \ref{sing+-}, for each $i$  the vector
$
v_{\la_{\lb,i}}\in U_q(\hat \n^-_{1\sms n}/\hat \n_{i+1\sms n}^-)v_{\la}
$
is $\g_{i+1\sms n}$-singular. Let  $M_{\la_{\lb,i}}=U_q( \n_{i+1\sms n}^-)v_{\la_{\lb,i}}$ denote the $U_q(\g_{i+1\sms n})$-Verma submodule in $M_\la$ generated by
$v_{\la_{\lb,i}}$.
In the similar way we define the $U_q(\g_{i+1\sms n})$-Verma submodule $M_{\la_{\kb,i}}^\star=v^\star_{\la_{\kb,i}} U_q( \n_{i+1\sms n}^+)$ in $M_\la^\star$ generated by a $\g_{i+1\sms n}$-singular vector
$
v^\star_{\la_{\kb,i}}\in v^\star_{\la} U_q(\hat \n^+_{1\sms n}/\hat \n_{i+1\sms n}^+)
$.

By construction, $\la_{\kb,0}=\la_{\lb,0}=\la$. Suppose that we have proved the equality
$\la_{\kb,i-1}=\la_{\lb,i-1}$  for some $i\in [1,n)$.  Then
$v_\la^\star \hat e(\kb)\hat f(\lb)v_\la$ can be presented as the matrix coefficient
$v_{\la_{\lb,i-1}}^\star \hat e(\kb_{i})\ldots \hat e(\kb_{n})\hat f(\lb_n)\ldots \hat f(\lb_i)v_{\la_{\lb,i-1}}$
of a cyclic paring between the $U_q(\g_{i+1\sms n})$-modules  $M_{\la_{\kb,i}}^\star $ and $M_{\la_{\lb,i}}$.
It is zero unless $\la_{\kb,i}=\la_{\lb,i}$. This is true for all $i\in [0,n]$, by induction on $i$.

The equalities $\la_{\kb,i}-\la_{\kb,i-1}=\la_{\lb,i}-\la_{\lb,i-1}$ for $i\in [1, n]$
translate to a triangular system of equations on the differences $k_{is}-l_{is}$:
namely, $\sum_{s=j}^n (k_{is}-l_{is})=0$ for all  $j=i,\ldots, n$.
It is immediate that $\kb_{i}=\lb_{i}$ for all $i\in [1,n]$ and therefore $\kb=\lb$.
\end{proof}
If follows that $U_q(\hat \n_{k\sms n}^-)v_{\la}$ is orthogonal to
$v_{\la}^\star U_q(\hat \n^+_{i\sms n}/\hat \n_{kn}^+)$ and
$U_q(\hat \n^-_{i\sms n}/\hat \n_{kn}^-)v_{\la}$ is orthogonal to
$v_{\la}^\star U_q(\hat \n_{k\sms n}^+)$ for all $i,k\in[1,n]$, $i< k$.
Calculation of (\ref{diagonal}) boils down to calculation
of the matrix coefficients
$$
v^\star_{\mu} \hat e(\lb_{k}) \hat f(\lb_{k}) v_{\mu} , \quad 1\leqslant k\leqslant n,
$$
where $v_\mu\in M_\la$ and $v^\star_{\mu}\in M^\star_{\la}$ are $\g_{kn}$-singular vectors.
This is done in the following section.
\section{The matrix coefficients}
Given a weight $\mu\in \h^*$, we put $\mu_i=(\mu,\al_i)$ and $\mu_{ij}=\mu_i+\ldots+\mu_j+j-i$,
assuming  $i\leqslant j\leqslant n$.
We adopt the convention that products  $\prod_{i=a}^b$  are
not implemented (formally set to $1$) once $a>b$.
For every $\lb\in \Tg$ and every $k\in [1,n]$ we define
$$
A_{\lb,k}(\mu)=\prod_{k+1\leqslant s\leqslant r\leqslant n}\>\prod_{i=0}^{l_{r}-1}[\mu_{sr}-i+l_{s-1}+1]_q,
\quad \lb_k=(l_n,\ldots,l_k).
$$
According to this definition, $A_{\lb,k}(\mu)$ actually depends on the $k$-th row  $\lb_k\in \Z_+^{n-k+1}$ of $\lb$.

\begin{lemma}
\label{unhat}
The matrix coefficient $v^\star_\la \hat e(\lb_{1}) \hat f(\lb_{1}) v_\la$ is equal to
$A_{\lb,1}(\la) v^\star_\la \hat e(\lb_{1}) f(\lb_{1}) v_\la
$.

\end{lemma}
\begin{proof}
The element $\hat f(\lb_{1})$ is a monomial in the dynamical root vectors $\hat f_{1\sms m}$, where $m$ ranges from $1$ to $n$.
The  element $\hat f_{1\sms m}$ is a sum of monomials in  $f_1,\ldots, f_m$ with coefficients from the Cartan subalgebra.
Let us prove that only $f_1\ldots f_m$ survives in each copy of $\hat f_{1\sms m}$. The other monomials, which are obtained by a permutation of the simple root vectors
$f_i$, vanish in the matrix coefficient. Suppose we have replaced all $\hat f_{1\sms m}$ with $f_1\ldots f_m\prod_{i=2}^m[h_{im}+1]_q$
 on the left of some factor $\hat f_{1\sms k}$ and denote the result by $\psi$, i.e.,
$f(\lb_{1}) v_\la=\psi\hat f_{1\sms k}\ldots v_\la$.
The element $\psi$ is a product of the monomials $f_1\ldots f_m$ with $m\geqslant k$, and
$f_1\ldots f_m\in \>{}\!\!^{\n^-_{2\sms m}}U$ by Lemma \ref{nilp}. This implies
 $\psi \n^-_{2\sms k}\subset \n^-_{2\sms n} U$.
Every monomial $\phi=f_{\si(1)}\ldots f_{\si(k)}$ entering $\hat f_{1\sms k}$ with $\si\not =\id$
belongs to $\n^-_{2\sms k}U$ by Lemma \ref{e1kf1m}.
Therefore, the vector $v^\star_\la \hat e_{\lb_1}\psi \phi \in v^\star_\la \hat e_{\lb_1}\n^-_{2\sms n}U$ is nil.

By this reasoning, we can consecutively  replace each $\hat f_{1\sms m}$ with $f_1\ldots  f_{m} \prod_{i=2}^m[h_{im}+1]_q$
factor by factor from left to right.
The Cartan coefficients  produce scalar multipliers, which gather to the overall factor $A_{\lb,1}(\la)$.
Finally, we replace each $f_1\ldots f_m$ with $f_{1\sms m}$ by a similar reasoning moving in the opposite direction, from  right to left.
\end{proof}

Next we calculate the matrix coefficient $v^\star_\la \hat e_{1\sms n}^{l} f_{1\sms n}^l v_\la$.
For all $k,m\in [1,n]$ such that $k\leqslant m$ we define polynomial functions $C_{km}\colon \h^*\to \C$ by
\be
\la\mapsto C_{km}(\la)=\prod_{i=k}^m[\la_{im}]_q.
\label{C-coefficient}
\ee
\begin{lemma}
The matrix coefficient $v^\star_\la \hat e_{1\sms n} f_{1\sms n}v_\la$
is equal to
$
C_{1\sms n}(\la).
$
\end{lemma}
\begin{proof}
We do induction on $n$. The statement for $n=1$ immediately follows from the defining relations. Suppose that $n>1$ and
present $f_{1\sms n}$ as $f_{1}f_{2\sms n}-qf_{2\sms n}f_{1}$. Observe that $v^\star_\la \hat e_{1\sms n}f_{2\sms n}f_{1}v_\la$
vanishes
since the vector $v^\star_\la\hat e_{1\sms n}$ is $\n_{2n}^-$-singular by Proposition \ref{sing+-}.
Now plug
$
\hat e_{1\sms n}=[h_{2\sms n}+1]_q\hat e_{2\sms n}e_{1}-[h_{2\sms n}]_qe_{1}\hat e_{2\sms n}
$
in $v^\star_\la \hat e_{1\sms n} f_{1\sms n}v_\la=
v^\star_\la  \hat e_{1\sms n}f_{1}f_{2\sms n}v_\la$ and push $f_{1}$ to the left. Observe that $f_1$ commutes with $\hat e_{2\sms n}$. The commutators of $f_1$ with the Cartan factors can be also neglected, as $f_1$ kills $v_\la^\star$.
We get for $v^\star_\la \hat e_{1\sms n} f_{1\sms n}v_\la$ the expression
$$
v^\star_\la ([h_{2\sms n}+1]_q\hat e_{2\sms n}[h_1]_q-[h_{2\sms n}]_q[h_1]_q\hat e_{2\sms n})f_{2\sms n}v_\la
=[\la_{1\sms n}]_qv^\star_\la \hat e_{2\sms n} f_{2\sms n}v_\la,
$$
since
$
([h_{2\sms n}+1]_q[h_1+1]_q-[h_{2\sms n}]_q[h_1]_q=[h_1+h_{2\sms n}+1]_q=[h_{1\sms n}]_q.
$
 Therefore,
$v^\star_\la \hat e_{1\sms n} f_{1\sms n}v_\la=[\la_{1\sms n}]_qv^\star_\la \hat e_{2\sms n} f_{2\sms n}v_\la=C_{1\sms n}(\la)$,
by the straightforward induction.
\end{proof}
In the matrix coefficient $v^\star_\la \hat e_{1\sms n}^{l-1} \hat e_{1\sms n}f_{1\sms n}^l v_\la$,
present the rightmost copy of
$\hat e_{1\sms n}$ as a sum of Chevalley monomials  $e_{\si(1)}\ldots e_{\si(n)}$, $\si\in S_n$, with coefficients
from $U_q(\h)$. By Lemma \ref{basics comms}, the generators  $e_i$ commute with  $f_{1\sms n}$ for all
$i\in[2,n-1]$. Therefore, non-zero contributions to the matrix coefficient
are made only  by the monomials
$$
\phi_1= e_{1} \ldots e_n,\quad
\phi_i= e_{i}\ldots e_n e_{i-1}\ldots e_1,
\quad
\phi_n= e_n \ldots e_1,
$$
where $i\in (1,n)$.
Let us calculate $\phi_i f_{1\sms n}^l v_\la $. We do it modulo  $\n^-_{2\sms n}Uv_\la$, which disappears when paired with $v^\star_\la \hat e_{1\sms n}^{l-1}$.

For every $i=1,\ldots, n$ and all $l\in \N$, define functions $D_{i,l}\colon \h^*\to \C$ by
\be
D_{1,l}(\la)
&=&q^{-l+1}[l]_q (-q)^{n-1} q^{\la_2+\ldots +\la_n}[\la_1]_q,
\nn\\
D_{i,l}(\la)&=&
q^{-l+1}[l]_q(-q)^{n-i}q^{-\la_1-\ldots -\la_{i-1}+\la_{i+1}+\ldots+\la_n}[\la_i]_q
, \quad i\in [2,n-1],
\nn\\
D_{n,l}(\la)&=&
q^{-l+1}[l]_qq^{l-1}q^{-\la_{1}-\ldots-\la_{n-1}}[\la_n-l+1]_q.
\nn
\ee

\begin{lemma}
The action of the monomials $\phi_i$, $i\in [1,n]$, on the vectors
$f_{1\sms n}^l v_\la$, $l \in \N$,  is given  by
$
\phi_i  f_{1\sms n}^l v_\la = D_{i,l}(\la) f_{1\sms n}^{l-1} v_\la \mod \n^-_{2\sms n}Uv_\la.
$
\end{lemma}
\begin{proof}
Assuming $i\in [2,n]$, present $\phi_i$ as $\phi_{i}'e_1$, where $\phi_i'\in U_q(\n^+_{2\sms n})$.
Observe that the relation $f_{2\sms n}f_{1\sms n}=qf_{1\sms n} f_{2\sms n}$ easily follows from
 Lemma \ref{basics comms}. Along with the relation $[e_1,f_{1\sms n}]=f_{2\sms n}q^{-h_1}$
 from the same lemma, this yields
$$
\phi_{i}f_{1\sms n}^{l} v_\la
=\phi_{i}'f_{2\sms n}q^{-h_1}f_{1\sms n}^{l-1} v_\la+\phi_{i}'f_{1\sms n}f_{2\sms n}q^{-h_1}f_{1\sms n}^{l-2} v_\la+\ldots
=[l]_qq^{-\la_1}\phi_{i}'f_{2\sms n}f_{1\sms n}^{l-1} v_\la.
$$
Present $\phi_i'$ as $e_{s^i_1}\ldots e_{s^i_{n-1}}$ and write
$$
\phi_{i}'  f_{2\sms n}f_{1\sms n}^{l-1} v_\la=e_{s^i_1}\ldots e_{s^i_{n-1}}  f_{2n}f_{1\sms n}^{l-1} v_\la=[e_{s^i_1},\ldots [e_{s^i_{n-1}}, (f_{2n})(f_{1\sms n}^{l-1})]\ldots] v_\la.
$$
Applying the Leibnitz rule to these commutators, we can ignore $f_{1\sms n}^{l-1}$:
$$e_{s^i_1}\ldots e_{s^i_{n-1}}  f_{2\sms n}f_{1\sms n}^{l-1} v_\la=[e_{s^i_1},\ldots [e_{s^i_{n-1}},f_{2\sms n}]\ldots] f_{1\sms n}^{l-1} v_\la+\ldots
$$
The omitted terms contain residual vectors coming from $f_{2\sms n}$. They lie in $\n^-_{2\sms n}U$ and vanish in
the matrix coefficient.
Modulo  $\n^-_{2\sms n}U$, Lemma \ref{basics comms} yields
$$
\phi_{i}f_{1\sms n}^{l} v_\la=
q^{-l'}[l]_qq^{\dt_{in}l'}(-q)^{n-i}q^{-\la_1-\ldots -\la_{i-1}+\la_{i+1}+\ldots+\la_n}[\la_i-\dt_{in}l']_qf_{1\sms n}^{l-1} v_\la
, \quad i\in [2,n],
$$
where $l'=l-1$.  This proves the statement for $\phi_i$, $i\in [2,n]$.

Consider the remaining case of $\phi_1$.
Using the relation $[e_n,f_{1\sms n}]=-qf_{1\sms n-1}q^{h_n}$
and the relation $f_{1\sms n-1}f_{1\sms n}=q^{-1}f_{1\sms n} f_{1\sms n-1}$  from Lemma \ref{basics comms}, we get
$$
\phi_{1}f_{1\sms n}^{l} v_\la
=-q [l]_qq^{\la_n}e_1\ldots e_{n-1}f_{1\sms n-1}f_{1\sms n}^{l-1} v_\la
=(-q)^{n-1}[l]_q q^{\la_2+\ldots +\la_n} e_1f_{1}f_{1\sms n}^{l-1} v_\la.
$$
We have used $[e_i,f_{1\sms n}]=0$ for $i\in [2,n-1]$ in this calculation. Further,
$$
 e_1f_{1}f_{1\sms n}^{l-1} v_\la
 =
 [\la_1-l']_q f_{1\sms n}^{l-1} v_\la
 +
[l']_qq^{-\la_1}f_1f_{2\sms n}f_{1\sms n}^{l-2} v_\la.
$$
We replace the product $f_1f_{2\sms n}$ with $f_{1\sms n}$, since  the calculation is done modulo $\n^-_{2\sms n}U$. Thus,
$$
\phi_1 f_{1\sms n}^{l} v_\la
=[l]_q (-q)^{n-1} q^{\la_n+\ldots +\la_2}\Bigl([\la_1-l']_q
 +
[l']_qq^{-\la_1}\Bigr)f_{1\sms n}^{l-1} v_\la \mod \n^-_{2\sms n}U.
$$
Notice that the factor in the brackets is equal to
$
[\la_1-l']_q +[l']_qq^{-\la_1}=q^{-l'}[\la_1]_q.
$
This completes the proof.
\end{proof}
The coefficients $D_{i,l}(\la)$ satisfy the reduction formulas
\be
D_{1,l}(\la)-q^{-l'}[l]_q (-1)^{n-1}q^{\la_{1\sms n}-l'}[l']_q&=&
q^{-l'}[l]_q D_{1,1}(\la-l'\al_{1\sms n}),
\label{red1}
\\
D_{i,l}(\la)&=&q^{-l'}[l]_q D_{i,1}(\la-l'\al_{1\sms n}),\quad i\in[2,n],
\label{red2.n}
\ee
which readily follow from their definition. As above, $l'=l-1$.
\begin{lemma}
\label{lemma_elementary}
For all $l\in \N$,
\be
 v^\star_\la \hat e_{1\sms n}^{l} f_{1\sms n}^l v_\la
&=&[l_q]!\prod_{i=0}^{l-1}[\la_{1\sms n}-i]_q[\la_{2\sms n}-i]_q\ldots [\la_{nn}-i]_q .
\label{elementary}
\ee
\end{lemma}
\begin{proof}
Let us calculate the vector $\hat e_{1\sms n}f_{1\sms n}^l v_\la$ modulo $\n^-_{2\sms n}Uv_\la$.
Consider the presentation  $\hat e_{1\sms n}=\sum_{i=1}^na_i(h)\phi_i + \ldots$ with  suppressed
Chevalley monomials from $U\n^+_{2\sms n-1}$.
They  make zero contribution to the vector $\hat e_{1\sms n}f_{1\sms n}^lv_\la$, because
$\n^+_{2\sms n-1}$ commutes with  $ f_{1\sms n}$ and kills $f_{1\sms n}^{l}v_\la$, by Lemma \ref{basics comms}.
We need the explicit expression only for $a_1(h)=(-1)^{n-1}\prod_{i=2}^n[h_{in}]_q$,
which is readily found from the definition of $\hat e_{1\sms n}$.
We replace $\hat e_{1\sms n}$ with its  specialization at the weight  $\la-l'\al_{1\sms n}$ and write
\be
\hat e_{1\sms n}f_{1\sms n}^lv_\la = \sum_{i=1}^n a_{i}(\la-l'\al_{1\sms n})D_{i,l}(\la)f_{1\sms n}^{l-1}v_\la\mod \n^-_{2\sms n}Uv_\la.
\label{ef^l}
\ee
Observe that $\sum_{i=1}^na_i(\mu)D_{i,1}(\mu)=C_{1\sms n}(\mu)$ for all $\mu\in \h^*$.
For higher $l$, the coefficient $\sum_{i=1}^n a_{i}(\la-l'\al_{1\sms n})D_{i,l}(\la)$
before $f_{1\sms n}^l v_\la$ in (\ref{ef^l})
is found to be
\be
&&\sum_{i=2}^n a_{i}(\la-l'\al_{1\sms n})D_{i,l}(\la)
+a_{1}(\la-l'\al_{1\sms n})D_{1,l}(\la)
\nn\\
&=&q^{-l'}[l]_q\sum_{i=2}^n a_{i}(\la-l'\al_{1\sms n})D_{i,1}(\la-l'\al_{1\sms n})+a_{1}(\la-l'\al_{1\sms n})D_{1,l}(\la)
\nn\\
&=&q^{-l'}[l]_qC_{1\sms n}(\la-l'\al_{1\sms n})
+a_{1}(\la-l'\al_{1\sms n})\bigl(D_{1,l}(\la)-q^{-l'}[l]_q D_{1,1}(\la-l'\al_{1\sms n})\bigr).
\nn
\ee
We have used the reduction formulas (\ref{red2.n}) in the second equality.
Plug in here the expressions
$$
C_{1\sms n}(\la-l'\al_{1\sms n})=[\la_{1\sms n}-2l']_q \prod_{k=2}^n[\la_{kn}-l']_q,
\quad
a_1(\la-l'\al_{1\sms n})=(-1)^{n-1}\prod_{k=2}^n[\la_{kn}-l']_q,
$$
and  the expression for the difference $D_{1,l}(\la)-q^{-l'}[l]_q D_{1,1}(\la-l'\al_{1\sms n})$ from (\ref{red1}).
This gives the coefficient before $f_{1\sms n}^l v_\la$ in (\ref{ef^l}). It is divisible
by $q^{-l'}[l]_q\prod_{k=2}^n[\la_{kn}-l']_q $,
which can be factored out.
The remaining factor is
$$
[\la_{1\sms n}-2l']_q +[l']_q q^{\la_{1\sms n}-l'}=\frac{q^{\la_{1n}-2l'}-q^{-\la_{1n}+2l'} +(q^{l'}-q^{-l'}) q^{\la_{1n}-l'}}{q-q^{-1}}
=q^{l'}[\la_{1\sms n}-l']_q.
$$
Combining this with the multiplier $q^{-l'}[l]_q\prod_{k=2}^n[\la_{kn}-l']_q $ we obtain the recurrent formula
$\langle v^\star_\la \hat e_{1\sms n}^{l} f_{1n}^l v_\la \rangle=[l]_q\prod_{k=1}^n[\la_{kn}-l']_q \langle v^\star_\la \hat e_{1n}^{l-1} f_{1n}^{l-1} v_\la \rangle$. Induction on $l$ completes the proof.
\end{proof}

From now on we understand by $\hat{e}(\lb)$  the {\em normally} ordered PBW monomial.
To proceed with the calculation of matrix coefficients of the cyclic Shapovalov pairing, we fix another
ordering on the positive dynamical root vectors: we define
$$
\check{e}(\lb_k)=\hat e_{kn}^{l_{kn}} \ldots \hat e_{kk}^{l_{kk}},
\quad
\check{e}(\lb)=\check{e}(\lb_1)\ldots \check{e}(\lb_n).
$$
In the last section we demonstrate that $\check{e}(\lb)=\hat e(\lb)$, but the proof of this nontrivial fact is indirect
and based on the knowledge of the matrix coefficients $v_\la^\star \check{e}(\lb_k)\hat f(\lb_k)v_\la$.

\begin{lemma}
\label{in_row_reduction}
Put $\lb_1=(l_n,\ldots,l_1)\subset \Z_+^n$.
Then
$
v^\star_\la \check{e}(\lb_1) f(\lb_1) v_\la
=\prod_{k=1}^n v^\star_\la \hat e_{1\sms k}^{l_k}f_{1\sms k}^{l_k} v_\la.
$
\end{lemma}
\begin{proof}
The above factorization of the matrix coefficient is a consequence of the formula
$$
v^\star_\la \hat e_{1\sms n}^{l_n}\ldots \hat e_{1\sms k}^{l_k} f_{1\sms n}^{l_n}\ldots  f_{1\sms k}^{l_k} v_\la
=v^\star_\la (\hat e_{1\sms n}^{l_n}\ldots \hat e_{1\sms k+1}^{l_{k+1}})(f_{1n}^{l_n}\ldots f_{1\sms k+1}^{l_{k+1}}) \hat e_{1\sms k}^{l_k} f_{1\sms k}^{l_k} v_\la,
$$
which holds true for all  $k\in [1,n]$. Let us prove it.
Denote by $\psi$ the product $f_{1\sms n}^{l_n}\ldots f_{1\sms k+1}^{l_{k+1}}$.
It is sufficient to show that $\hat e_{1\sms k}$ commutes with $\psi$ modulo $\n^-_{2\sms n}U$,
as $\n^-_{2\sms n}$ annihilates $v^\star_\la \hat e_{1\sms n}^{l_n}\ldots \hat e_{1\sms k}^{l_k-i-1}$.
Let $\nu$ denote the weight of this vector  and let $\tilde e_{1\sms k}\in U_q(\n^+_{1\sms k})$ be
the specialization of $\hat e_{1\sms n}$ at $\nu$.
It  follows from Lemma \ref{e1kf1m} and Lemma \ref{nilp} that
$
[\tilde e_{1\sms k} ,\psi]\in \n^-_{2\sms n}U.
$
Therefore, we can replace
$ \hat e_{1\sms k} \psi$
with
$\psi \tilde e_{1\sms k}\mod\n^-_{2\sms n}U$.
Finally, observe that the Cartan coefficients of $\hat e_{1\sms k}$ are confined within
$U_q(\h_{2\sms k})$ and consequently commute with $\psi$. Therefore, $\psi\tilde e_{1\sms k}$
can be replaced with $\psi\hat e_{1\sms k}$ modulo $\n^-_{2\sms n}U$.

To finish the proof, observe that
$\hat e_{1\sms k}^{l_k} f_{1\sms k}^{l_k} v_\la=\langle v^\star_\la\hat e_{1\sms k}^{l_k}, f_{1\sms k}^{l_k} v_\la\rangle\times v_\la$. Varying $k$ from $1$ to $n$
we prove the  factorization of $v^\star_\la \check{e}(\lb_1) f(\lb_1) v_\la$.
\end{proof}
So far in this section we dealt with the matrix coefficients $v^\star_\la \check{e}(\lb_1) f(\lb_1) v_\la$,
i.e. of the form
$v_\la^\star U_q(\hat \n^+_{1\sms n}/\hat \n^+_{2\sms n})U_q(\hat \n^-_{1\sms n}/\hat \n^-_{2\sms n})v_\la$.
Upon obvious modifications, these results hold true for $v^\star_\mu \check{e}(\lb_k) f(\lb_k) v_\mu$,
for any $k\in[1,n]$ and $v^\star_\mu \in M_\mu^\star$,  $v_\mu \in M_\mu$ being $\g_{kn}$-singular
vectors.

\begin{corollary}
Suppose that $v_\mu\in M_\la$ and $v_\mu^\star \in M_\la^\star$ are $\g_{kn}$-singular vectors of weight $\mu$.
Then the  matrix coefficient
$ v^\star_\mu \check e(\lb_k) \hat f(\lb_k) v_\mu $
is equal to
\be
[l_{k}]_q!\ldots [l_{n}]_q!\prod_{k\leqslant s\leqslant r\leqslant n}\prod_{i=0}^{l_{r}-1}[\mu_{sr}-i]_q
\times
\prod_{k+1\leqslant s\leqslant r\leqslant n}\>\>\prod_{i=0}^{l_{r}-1}[\mu_{s r}-i+l_{s-1}+1]_q
\> v^\star_\mu v_\mu,
\ee
where $l_r=l_{kr}$, $r=k,\ldots, n$.
\label{row_mat_coef}
\end{corollary}
\begin{proof}
Replacement of $\hat f(\lb_k)$ with $f(\lb_k)$ yields a scalar multiplier $A_{\lb,k}(\mu)$, as explained
by Lemma \ref{unhat}; hence the last product. Factorization of $ v^\star_\mu \check{e}(\lb_k) f(\lb_k) v_\mu $
is established by  Lemma \ref{in_row_reduction} and Lemma \ref{lemma_elementary}; hence the first product with the factorials.
\end{proof}
Denote the matrix coefficients from Corollary \ref{row_mat_coef} by $B_{\lb_{k}}(\mu)$
and define
$$
B_{\lb}(\la)=B_{\lb_1}(\la_{\lb,0})\ldots B_{\lb_n}(\la_{\lb,n-1}),
$$
where the weights $\la_{\lb,i}$ are introduced in (\ref{partial_weight}).

\begin{thm}
The matrix coefficient
$
v_\la^\star \check{e}(\kb)
\hat f(\lb)v_\la
$
is equal to
$
\dt_{\kb,\lb}B_{\lb}(\la)
$.
\label{main}
\end{thm}
\begin{proof}
The Kronecker symbol is justified in Proposition \ref{diagonal}.
Further, let $v_{\la_{\lb,i}}\in M_\la$ and $v_{\la_{\lb,i}}^\star\in M_\la^\star$, $i\in [0,n)$, be
the vectors defined in (\ref{partial_sing}),
where the positive PBW monomial is ordered as $\check e_{\lb}$.
Due to Lemma \ref{sing+-}, the matrix coefficient
$
v_\la^\star \check{e}(\lb)
\hat f(\lb)v_\la
$ factorizes to
$$
v_{\la_{\lb,n-1}}^\star \check{e}(\lb_n)\hat f(\lb_n)v_{\la_{\lb,n-1}}
=B_{\lb_n}(\la_{\lb,n-1}) v_{\la_{\lb,n-1}}^\star v_{\la_{\lb,n-1}}=\ldots =
B_{\lb_n}(\la_{\lb,n-1})\ldots B_{\lb_1}(\la_{\lb,0}) v_{\la_{\lb,0}}^\star v_{\la_{\lb,0}},
$$
where $v_{\la_{\lb,0}}^\star v_{\la_{\lb,0}}=v_{\la}^\star v_{\la}=1$.
At every step $k\in [1,n]$ we apply Corollary \ref{row_mat_coef} in order to calculate the matrix coefficient
$
v_{\la_{\lb,k-1}}^\star \check{e}(\lb_k)\hat f(\lb_k)v_{\la_{\lb,k-1}}
$
with $\check{e}(\lb_k),\hat f(\lb_k) \in U_q(\g_{kn})$
and the $\g_{kn}$-singular vectors $v_{\la_{\lb,k-1}}^\star, v_{\la_{\lb,k-1}}$
of weight $\mu=\la_{\lb,k-1}$.
\end{proof}
\begin{corollary}
Suppose the weight $\la$ is such that $B_{\lb}(\la)\not=0$ for all $\lb\in \Tg$.
Then the system $\{\hat f(\lb) v_\la\}_{\lb\in\Tg}$
forms a basis in $M_\la$ and $ \{\frac{1}{B_{\lb}(\la)}v_\la^\star\check e(\lb)\}_{\lb\in\Tg}$ is its dual
 basis in $M_\la^\star$. The formal sum $\sum_{\lb}\frac{1}{B_{\lb}(\la)}\hat f({\lb})v_\la\tp v_\la^\star \check{e}({\lb}) \in M_\la\tp M_\la^\star$
is the inverse of the cyclic Shapovalov pairing.
\end{corollary}
\begin{proof}
Both  systems $\{\hat f(\lb)v_\la\}$,  $\{v_\la^\star\check e({\lb})\}$  have the same number of vectors of a given weight
as their standard PBW counterparts. Their Gram matrix with respect to the Shapovalov pairing is non-degenerate,
provided all $B_{\lb}(\la)\not =0$. Therefore, for such $\la$, $\{\hat f({\lb})v_\la\}$ is a basis
in $M_\la$ and $\{\frac{1}{B_{\lb}}v_\la^\star\check e({\lb})\}$ is its dual in $M_\la^\star$, according to
Theorem \ref{main}.
\end{proof}
The classical version of these results is straightforward. One should pass to the algebra $U_\hbar(\g)$ and
take the zero fiber $\!\mod \hbar$. This operation converts $[x]_q$ into $x$ for any indeterminate $x$. The
classical version of the dynamical root vectors and the formulas for the matrix coefficients are immediate.

Recall from \cite{Sh,DCK} that the Shapovalov form on $M_\la$ is invertible if and only if
$q^{2(\la+\rho,\al)}\not \in q^{2\N}$ (respectively, $(\la,\al)+(\rho,\al)\not \in \N$ for $U(\g)$)
for all $\al\in R^+$.
In our notation, this criterion translates to  $q^{2\la_{ij}}\not \in q^{2\Z_+}$ (respectively, $\la_{ij}\not \in \Z_+$)
for all $i,j$ such that $i\leqslant j$.
On the other hand, one can easily see that the set of zeros of $B_{\lb}(\la)$, $\lb \in \Tg$, is larger although
contained in the union $\cup_{\al\in R^+}\{\la|q^{2(\la,\al)}\in q^{2\Z}\}$ (in the union of integer hyperplanes
$(\la,\al)\in \Z$ in the classical case). Therefore, the system $\hat f(\lb)v_\la$, $\lb \in \Tg$, fails to be a basis for special values of weights.
We consider this effect in a more detail on the example of $\s\l(3)$ in the last section.
\begin{example}
\label{exmp1k1m} \em
Here is an example which will play a role in the next section.
We   need the explicit expression for the
 matrix coefficient $v_\la^\star \hat e_{1\sms m}\hat e_{1\sms k}\hat f_{1\sms m}\hat f_{1\sms k}v_\la$, $k<m$,
which is
\be
\prod_{j=2}^k(\la_{j\sms k}+1)
\prod_{j=2}^k[\la_{j\sms m}+1]_q[\la_{k+1\sms m}+2]_q\prod_{j=k+2}^m[\la_{j\sms m}+1]_qC_{1\sms k}(\la)C_{1\sms m}(\la),
\label{1k1m}
\ee
according to the general formula. As usual, the products are present only if the lower bounds do not exceed the upper bounds.
The products before $C_{1\sms k}(\la)C_{1\sms m}(\la)$ results from the transition
$\hat f_{1\sms k}\to f_{1\sms k}$, $\hat f_{1\sms m}\to f_{1\sms m}$.
The Cartan coefficients $[h_{i\sms m}+1]_q$ from $\hat f_{1\sms m}$ commute with $f_{1\sms k}$ unless $i= k+1$,
while $[h_{k+1\sms m},f_{1\sms k}]=f_{1\sms k}$. This accounts for $2$ in the corresponding factor.
 \end{example}
\section{Contravariant Shapovalov form }
\label{SecCShF}
In this section we refine the obtained results and show that the dual bases in $M_\la^\star$ and $M_\la$
give rise to an orthogonal basis for the contravariant form on $M_\la$.
The key step is to prove that the dynamical positive (negative) root vectors commute within each row. This facilitates the
equalities $\check{e}(\lb_k)=\hat e(\lb_k)$ for all $k\in [1,n]$ and $\check{e}(\lb)=\hat e(\lb)$ for all $\l\in \Tg$.

We start with the following simple case, which will be the base for a further induction.
\begin{lemma}
For all $m\in [1,n]$, one has  $[f_1,\hat f_{1\sms m}]=0$ and
 $[e_1,\hat e_{1\sms m}]=0$.
 \label{row_com_1}
\end{lemma}
\begin{proof}
This is an immediate consequence of the Serre relation:
\be
f_1f_{1.m}&=&
f_1(f_1\hat f_{2.m}[h_{2.m}+1]_q-\hat f_{2.m}f_1[h_{2.m}]_q)
\nn\\
&=&
f_1\hat f_{2.m}f_1([2]_q[h_{2.m}+1]_q-[h_{2.m}]_q)
-\hat f_{2.m}f_1^2[h_{2.m}+1]_q
=f_{1.m}f_1,
\nn
\ee
since the difference in the brackets is equal to $[h_{2.m}+2]_q$. Applying $\omega$ to
$[f_1,\hat f_{1\sms m}]=0$ gives  $[e_1,\hat e_{1\sms m}]=0$.
\end{proof}
One can directly check $[\hat f_{12},\hat f_{1m}]=0$ for all $m$ via a more cumbersome calculation.
We have not found a direct general proof, apart from the above simplest cases, and
use a roundabout approach based
on already obtained results. Namely, we will show that positive dynamical PBW system vanishes when
paired with the
element $[\hat f_{1k},\hat f_{1m}]v_\la$ for all $\la$. Since it is a basis in $M_\la^\star$ and the pairing
is non-degenerate for generic $\la$,
that will be sufficient to prove the equality $[\hat f_{1k},\hat f_{1m}]=0$.
\begin{propn}
\label{row-commute}
For every $i\in [1,n)$, the algebra $U_q(\n^\pm_{in}/\n^\pm_{i+1\sms n})$ is commutative.
\end{propn}
\begin{proof}
It is sufficient to check only  $U_q(\n^-_{in}/\n^-_{i+1\sms n})$, thanks to the Chevalley involution.
This algebra is generated by $\hat f_{ik}$, $k=i,\ldots, n$. To prove the equality
$[\hat f_{ik},\hat f_{im}]=0$, we do induction on  $k-i$, where $k$ is assumed to be less than $m$.

The case $k-i=0$ is already established by Lemma \ref{row_com_1}. For higher $k$ and $m\geqslant 3$, let us
prove that the vector $[\hat f_{ik},\hat f_{im}]v_\la\in M_\la$ is annihilated by $v_\la^\star U_q(\hat \n^+_{1n})$
for all $\la$. It suffices to restrict to $v_\la^\star U_q(\hat \n^+_{in})$, because
$[\hat f_{ik},\hat f_{im}]v_\la\in U_q(\hat \n^-_{in})v_\la$. Therefore, we can assume $i=1$.
Since $v_\la^\star \hat U_q(\n^+_{2\sms n})[\hat f_{1k},\hat f_{1m}]v_\la=0$,
we can restrict to $v_\la^\star U_q(\hat \n^+_{1n}/\hat \n^+_{2n})$.
By weight arguments, it is sufficient to calculate the matrix element
$
v_\la^\star \hat e_{1\sms m}\hat e_{1\sms k}\hat f_{1k}\hat f_{1\sms m}v_\la
$
and check it against
$
v_\la^\star \hat e_{1\sms m}\hat e_{1\sms k}\hat f_{1\sms m}\hat f_{1\sms k}v_\la,
$
which is given in Example \ref{exmp1k1m}.

Plugging the expression $\hat f_{1\sms k}=(f_1\hat f_{2\sms k}[h_{2\sms k}+1]_q-\hat f_{2\sms k}f_1[h_{2\sms k}]_q)$ in
$v_\la^\star \hat e_{1\sms m}\hat e_{1\sms k}\hat e_{1\sms k}\hat f_{1\sms m}v_\la$ we  get $[\la_{2\sms k}+1]_qv_\la^\star \hat e_{1m}\hat e_{1k}f_1\hat f_{2k}\hat f_{1m}v_\la$, since the second term makes zero contribution.
Developing $\hat e_{1\sms k}$ in the similar way we continue to
\be
v_\la^\star \hat e_{1\sms m}\hat e_{1\sms k}\hat e_{1\sms k}\hat f_{1\sms m}v_\la
&=&
[\la_{2\sms k}+1]_qv_\la^\star \hat e_{1\sms m}([h_{2\sms k}+1]_q\hat e_{2\sms k}e_1-[h_{2\sms k}]_q  e_1\hat e_{2\sms k})f_1\hat f_{2\sms k}\hat f_{1\sms m}v_\la.
\label{[k,m]}
\ee
Observe that $h_{2\sms k}$ commutes with $\hat e_{1\sms m}$. The second term gives $-[\la_{2\sms k}+1]_q[\la_{2\sms k}]_q$ times
\be
v_\la^\star \hat e_{1\sms m}e_1f_1\hat e_{2\sms k}\hat f_{2\sms k}\hat f_{1\sms m}v_\la
&=&C_{2k}(\la)\prod_{j=3}^k[\la_{jk}+1]_qv_\la^\star \hat e_{1m}e_1\hat f_{1m}f_1v_\la.
\nn
\ee
In accordance with our
convention, the product is replaced by $1$ if $k=2$.
We have used the fact that $\hat f_{1\sms m}v_\la$ is $\n^+_{2k}$-singular and
$\hat e_{2\sms k}\hat f_{2\sms k}\hat f_{1\sms m}v_\la=\langle v_\la^\star\hat e_{2\sms k},\hat f_{2\sms k}\rangle \hat f_{1\sms m}v_\la$. Also, we have applied Lemma \ref{row_com_1}.
The matrix coefficient
in the right-hand side is standard, and can be specialized from the general formula (\ref{1k1m}).
 The  contribution
of the second term in (\ref{[k,m]}) is
\be
-[\la_{2\sms k}]_q[\la_1]_q[\la_{2\sms m}+2]_q\prod_{j=2}^k[\la_{jk}+1]_q\prod_{j=3}^m[\la_{jm}+1]_qC_{2\sms k}(\la)C_{1\sms m}(\la).
\label{cc2}
\ee
Here we have used $C_{2\sms k}(\la-\al_{1\sms m})=C_{2\sms k}(\la)$, which is true for $k<m$.

The first term  in (\ref{[k,m]})  gives $[\la_{2\sms k}+1]_q^2$ times
$$
v_\la^\star \hat e_{1\sms m}\hat e_{2\sms k}e_1f_1\hat f_{2\sms k}\hat f_{1\sms m}v_\la=[\la_1]_qv_\la^\star \hat e_{1\sms m}\hat e_{2\sms k}\hat f_{2\sms k}\hat f_{1\sms m}v_\la+
v_\la^\star \hat e_{1\sms m}\hat e_{2\sms k}f_1e_1\hat f_{2\sms k}\hat f_{1\sms m}v_\la.
$$
The first matrix coefficient is standard and can be extracted from Theorem \ref{main}. The total contribution of this term to
(\ref{[k,m]}) is
\be
[\la_{2\sms k}+1]_q[\la_1]_q\prod_{j=2}^k[\la_{jk}+1]_q\prod_{j=2}^m[\la_{jm}+1]_qC_{2\sms k}(\la)C_{1\sms m}(\la),
\label{cc11}
\ee
since $C_{2\sms k}(\la-\al_{1\sms m})=C_{2\sms k}(\la)$.
Let us compute the matrix coefficient
$v_\la^\star \hat e_{1\sms m}\hat e_{2\sms k}f_1e_1\hat f_{2\sms k}\hat f_{1\sms m}v_\la
=
v_\la^\star \hat e_{1\sms m}f_1\hat e_{2\sms k}\hat f_{2\sms k}e_1\hat f_{1\sms m}v_\la
$.
With the use of the right equalities from Proposition \ref{key},
we find it equal to
$$
v_\la^\star \hat e_{1\sms m}f_1\hat e_{2\sms k}\hat f_{2\sms k}e_1\hat f_{1\sms m}v_\la
=
[\la_{1\sms m}]_q^2v_\la^\star \hat e_{2\sms m}\hat e_{2\sms k}\hat f_{2\sms  m}\hat f_{2\sms k}v_\la,
$$
by the induction assumption. The total contribution of this term to (\ref{[k,m]})
is
\be
[\la_{2\sms k}+1]_q^2[\la_{1\sms m}]_q^2\prod_{j=3}^k[\la_{j\sms k}+1]_q
\prod_{j=3}^{k}[\la_{j\sms m}+1]_q[\la_{k+1\sms m}+2]_q\prod_{j={k+2}}^m[\la_{j\sms m}+1]_q
C_{2\sms k}(\la)C_{2\sms m}(\la),
\label{cc12}
\ee
where again the convention about the products is in effect.

The matrix coefficient (\ref{[k,m]}) comprises (\ref{cc2}-\ref{cc12}), which contain
the common factor
$F_1= \prod_{j=2}^k[\la_{j\sms k}+1]_q
\prod_{j=3}^k[\la_{j\sms m}+1]_q\prod_{j=k+2}^m[\la_{j\sms m}+1]_qC_{2\sms k}(\la)C_{1\sms m}(\la)$.
Division by $F_1$ gives
$$
[\la_{2\sms k}+1]_q[\la_{1\sms m}]_q[\la_{k+1\sms m}+2]_q
+[\la_1]_q[\la_{k+1\sms m}+1]_q[\la_{2\sms k}+1]_q
[\la_{2\sms m}+1]_q
-[\la_1]_q[\la_{k+1\sms m}+1]_q[\la_{2\sms k}]_q[\la_{2\sms m}+2]_q,
$$
which we denote by $F_2$.
The last two terms  produce
$
[\la_1]_q[\la_{1\sms m}-\la_{1\sms k}]_q[\la_{k+1\sms m}+2]_q
$
since $\la_{2\sms m}-\la_{2\sms m}+1=\la_{k+1\sms m}+2$
and  $\la_{k+1\sms m}+1=\la_{1\sms m}-\la_{1\sms k}$.
Combine this with the first term in $F_2$ having made the replacement $\la_{2\sms k}+1=\la_{1\sms k}-\la_1$.
This gives $F_2=[\la_{k+1\sms m}+2]_q[\la_{1\sms k}]_q[\la_{2\sms m}+1]_q.$
Now one can see that the matrix coefficient (\ref{[k,m]}), which is equal to $F_1F_2$, is identical to the matrix coefficient
from Example \ref{exmp1k1m}. This completes the proof.
\end{proof}

In conclusion, let us turn to the contravariant Shapovalov form $M_\la$ defined through
the Chevalley involution $\omega$.
Consider the linear isomorphism $\theta\colon M_\la \to M_\la^\star$, $\theta\colon uv_\la\mapsto v_\la^\star\omega(u)$,
where $u\in U_q(\n^-)$. Obviously, $\theta(xv)=\theta(v)\omega(x)$ for all $x\in U_q(\g)$ and $v\in M_\la$. The contravariant form
 is defined on $M_\la$ through the composition $M_\la\tp M_\la\stackrel{\theta \tp\id}{\longrightarrow} M_\la^\star\tp M_\la \to \C$,
 where the right arrow is the cyclic Shapovalov pairing.

\begin{corollary}
The system $\frac{1}{\sqrt{B_{\lb}(\la)}}\hat f(\lb)v_\la$, $\lb\in \Tg$, forms an orthonormal basis
with respect to the contravariant form on
the Verma module $M_\la$, provided $B_{\lb}(\la)\not=0, \forall\lb\in \Tg$.
\end{corollary}
\begin{proof}
Follows from Theorem \ref{main} and Proposition \ref{row-commute}, since
 $\theta\bigl(\hat f(\lb)v_\la\bigr)=v_\la^\star \hat e(\lb)=v_\la^\star \check{e}(\lb)$.
\end{proof}

\section{The case of $\g=\s\l(3)$}
We illustrate Theorem \ref{main} on the simple example of $\g=\s\l(3)$ reproducing the key steps of the
calculations.  Now $\n^-=\Span\{f_1,f_2,f_{12}\}$ and $\n^+=\Span\{e_1,e_2,e_{12}\}$,
where $f_{12}=f_1f_2-q f_2f_1$ and $e_{12}=e_2e_1-q e_1e_2$.
The dynamical root vectors $\hat e_{12}$ and $\hat f_{12}$ are
$\hat e_{12}=[h_2+1]_qe_2e_1-[h_2]_q e_1e_2$ and $\hat f_{12}=f_1f_2[h_2+1]_q- f_2f_1[h_2]_q$.
For all $l,m \in \Z_+$ the vector $\hat f_{12}^lf_1^mv_\la \in M_\la$ is annihilated by $e_2$, and
similarly $v_\la^\star \hat e_{12}^l e_1^m \in M_\la^\star  $ is annihilated by $f_2$.
This readily implies
$$
v_\la^\star  e_1^s \hat e_{12}^r  e_2^pf_2^k\hat f_{12}^lf_1^mv_\la
=\dt_{pk}[k]_q!\prod_{i=0}^{k-1}[\la_2-i+m-l]_q
v_\la^\star \hat e_{12}^r e_1^s \hat f_{12}^lf_1^mv_\la
$$
(we use $e_1 \hat e_{12}= \hat e_{12}e_1$, by Lemma \ref{row_com_1}).
The matrix coefficient in the right-hand side is not zero only if
 $r(\al_1+\al_2)+s\al_1=l(\al_1+\al_2)+m\al_1$ or, equivalently,
$r=l$, $s=m$, in accordance with Proposition \ref{diagonal}.

In the matrix coefficient
$
v_\la^\star \hat e_{12}^l e_1^m \hat f_{12}^lf_1^mv_\la
$,
every factor $\hat f_{12}=[f_1,f_2]_q[h_2+1]_q+ f_2f_1q^{h_2+1}$
can be replaced with $[f_1,f_2]_q[h_2+1]_q$. This specialization of Lemma \ref{unhat} becomes immediate due to
the fact that $f_2f_1$ commutes with $f_1f_2$ and can be pushed to the left, where
$f_2$ kills $v_\la^\star \hat e_{12}^l e_1^m$.
This yields
$$
v_\la^\star  e_1^m \hat e_{12}^l\hat f_{12}^lf_1^mv_\la=
v_\la^\star \hat e_{12}^l e_1^m \hat f_{12}^lf_1^mv_\la=
\prod_{i=0}^{l-1}[\la_2-i+m+1]_q
v_\la^\star \hat e_{12}^l e_1^m f_{12}^lf_1^mv_\la.
$$
Pushing every copy of $e_1$ to the right produces zero contribution of the commutator $[e_1,f_{12}^l]$, as
the latter belongs to $f_2U$, as in Lemma \ref{in_row_reduction}. In the present case, this is a consequence of the commutation relations
 $[e_1,f_{12}]= f_2q^{-h_i}$ and $[f_2,f_{12}]_q=0$, see Lemma \ref{basics comms}.
This yields the factorization
$$
v_\la^\star \hat e_{12}^l e_1^m f_{12}^lf_1^mv_\la=
v_\la^\star \hat e_{12}^l f_{12}^l e_1^m f_1^mv_\la
=
[m]_q!\prod_{i=0}^{m-1}[\la_1-i]_qv_\la^\star \hat e_{12}^l f_{12}^lv_\la,
$$
as in Lemma \ref{in_row_reduction}.
Combining the above steps with the value of the matrix coefficient $v_\la^\star \hat e_{12}^l f_{12}^lv_\la=[l_q]!\prod_{i=0}^{l-1}[\la_{1\sms 2}-i]_q[\la_{2}-i]_q$ given by (\ref{elementary}) (in this simple case it can be easily computed directly)
  we get for $B_{\lb}(\la)=v_\la^\star e_1^m\hat e_{12}^l  e_2^kf_2^k\hat f_{12}^lf_1^mv_\la$ the formula
$$
B_{\lb}(\la)=[l]_q![m]_q![k]_q!
\prod_{i=0}^{k-1}[\la_2-i+m-l]_q
\prod_{i=0}^{l-1}[\la_2-i+m+1]_q\prod_{i=0}^{l-1}[\la_2-i]_q
\prod_{i=0}^{l-1}[\la_{12}-i]_q\prod_{i=0}^{m-1}[\la_1-i]_q,
$$
where $m=l_{11}$, $l=l_{12}$, and $k=l_{22}$.

In the standard basis, the inverse of Shapovalov form  is known to have entries with simple poles, \cite{O,ESt}.
Examining $B_{\lb}(\la)$ suggests the presence of second order zeros at $\la_2=0,\ldots, \min\{l-1,-m+k+l-2\}$, provided $l$
and $-m+k+l-1$ are positive. This example shows that the singularities of the form inverse
 are not necessarily simple
in the basis $\hat f(\lb),\hat e(\kb)$.

Consider the classical limit $q\to 1$.  The set of zeros of $B_{\lb}(\la)$ over all $\lb\in \Tg$ is the union of hyperplanes $\la_1\in \Z_+$, $\la_{12}\in \Z_+$, and $\la_2\in \Z$.
At the points $\la_2\in -\N$ the form is still invertible,  therefore the system $f_2^k \hat f_{12}^l f_1^mv_\la$ fails to be a basis.
Consider the automorphism of $U(\s\l(3))$ corresponding to the inversion $\al_1\leftrightarrow \al_2$
 of the Dynkin diagram. This automorphism produces an alternative system of dynamical roots,
with $\hat e_{12}=(h_1+1) e_1 e_2-h_1 e_2 e_1$ and
$\hat f_{12}=f_2 f_1(h_1+1) -f_1 f_2h_1$.
With the reversed ordering on thus defined root vectors, we obtain a dynamical PBW system yielding a basis in
 $M_\la^\star$ and $M_\la$, provided
$\la_1\not\in \Z$ and $\la_2,\la_{12}\not\in \Z_+$. One or another  system is
a basis for $\la_1,\la_2,\la_{12}\not \in \Z_+$, i.e. exactly where the Shapovalov form is non-degenerate.

\section{Singular vectors in $M_\la$.}
In this final section we use the dynamical PBW basis to construct singular vectors in $M_\la$.
\begin{lemma}
\label{Chevalley_basis}
Suppose $\phi_1, \phi_2\in U_q(\g_-)$ are non-zero elements of weight $-\bt \in -\Rm^+$ such that
$[e_\al,\phi]=0$ and $(\bt,\al)\not =0$ for some $\al \in \Pi^+$.
Then the vectors $f_{\al}\phi_1, \phi_2 f_{\al}\in U_q(\g_-)$ are linearly independent.
\end{lemma}
\begin{proof}
The root $\bt$ is a sum $\al_i+\ldots + \al_j$ for some $i<j$. Then $\al$ is either $\al_{i-1}$ or
$\al_{j+1}$. Assume that $\al=\al_{i-1}$ as the other case is similar.
By the PBW property, $f_{im}$, $m=i,\ldots,j$ are independent 
over  $U_q(\n_{i+1,j}^-)$ with respect to the right multiplication.
We can write $\phi_k=\sum_{m=i}^{j}f_{im}a_{k}^m$, $k=1,2$, where $a_{k}^m\in  U_q(\n_{i+1,j}^-)$. 
If $x,y\in \C$ are such that $x f_\al\phi_1=y\phi_2f_\al$, then
$$
x f_\al\sum_{m=i}^{j}f_{im}a_{1}^m =y\sum_{m=i}^{j}f_{im}a_{2}^mf_\al =y \sum_{m=i}^{j}q^{-1}f_\al f_{im}a_{2}^m
-y \sum_{m=i}^{j}q^{-1} f_{i-1,m}a_{2}^m,
$$
and $y q^{-1} a_{2}^m=0$, $x a_{1}^m =y q^{-1}a_{2}^m$ for all $m=i,\ldots, j$. Since $\phi_i\not =0$, this gives
$x=y=0$.
\end{proof}

\begin{corollary}
\label{dyn_not0}
For all $\al \in R^+$, the element  $\hat{f}_\al (\la)\in U_q(\g_-)$ is not vanishing at all $\la$.
\end{corollary}
\begin{proof}
We do induction on $\deg \hat{f}_\al$.
For $\deg \hat{f}_\al=1$ the statement is obvious. Suppose that $\al$ is presentable as $\al=\al_i+\bt$, where $\al_i\in \Pi^+$ and $\bt \in R^+$. Suppose we have proved that $\hat{f}_\bt(\la)\not = 0$ for some $\la$. Then the non-zero vectors $f_{\al_i}\hat{f}_\bt v_\la=f_{\al_i}\hat{f}_\bt(\la) v_\la$ and
$\hat{f}_\bt f_{\al_i} v_\la =\hat{f}_\bt(\la) f_{\al_i} v_\la $ are independent, by Lemma \ref{Chevalley_basis}.
Therefore, $\hat{f}_\al v_\la=0$ if and only if
$q^{2(\la+\rho,\bt)}=1=q^{2(\la+\rho,\bt)-2}$, which is impossible since $q^2\not=1$.
\end{proof}

The standard higher root vectors $f_{ij}\in U_q(\g)$ are known to satisfy  the identity
$
f_1 f_{2\sms n}^2=[2]_qf_{2\sms n}f_{1\sms n}-f_{2\sms n}^2f_1=0
$,
which easily follows from the Serre relations. Further we need its dynamical version.
\begin{lemma}
\label{Dyn_h_Serre}
One has
$
f_1\hat f_{2\sms n}^2-[2]_q\hat f_{2\sms n}
f_1\hat f_{2\sms n}+\hat f_{2\sms n}^2f_1=0
$.
\end{lemma}
\begin{proof}
We prove an equivalent identity $
f_1\hat f_{2\sms n}^2[h_{2\sms n}+1]_q=\hat f_{2\sms n}^2f_1[h_{2\sms n}-1]_q+[2]_q\hat f_{2\sms n}\hat f_{1\sms n}
$, whose right-hand side involves ordered PBW monomials.
 It is clear that $f_1\hat f_{2\sms n}^2v_\la$ is singular with respect to $\g_{3\sms n}$.
Therefore, it is a linear combination of PBW monomials in $\hat f_{2\sms 2},\ldots \hat f_{2\sms n},\hat f_{1\sms 2},\ldots, \hat f_{1\sms n}$ applied to $v_\la$. By weight arguments, we can write
$f_1\hat f_{2\sms n}^2v_\la=A\hat f_{2\sms n}^2f_1v_\la+B\hat f_{2\sms n}\hat f_{1\sms n}v_\la$ for some
scalars $A$, $B$. Pairing this equality with
$v_\la^\star \hat e_{1}\hat e_{2\sms n}^2$ and $v_\la^\star \hat e_{1\sms n}\hat e_{2\sms n}$ we get
$$
[\la_1]v_\la^\star \hat e_{2\sms n}^2 \hat f_{2\sms n}^2v_\la=Av_\la^\star \hat e_{1}\hat e_{2\sms n}^2 \hat f_{2\sms n}^2f_1v_\la,
\quad
[\la_{1\sms n}]_qv_\la^\star \hat e_{2\sms n}^2\hat f_{2\sms n}^2v_\la=Bv_\la^\star \hat e_{1\sms n}\hat e_{2\sms n} \hat f_{2\sms n}\hat f_{1\sms n}v_\la,
$$
where we have used Proposition \ref{key} in the right equality.
Comparison of the matrix coefficients yields
$
A=\frac{[\la_{2\sms n}-1]_q}{[\la_{2\sms n}+1]_q}
$
and
$
B= \frac{[2]_q}{[\la_{2\sms n}+1]_q}
$,
as required.
\end{proof}
\begin{corollary}
\label{[f_2n,f_1n]}
Put
$
\bar f_{1\sms n}=f_1\hat f_{2\sms n}[h_{2\sms n}+2]_q-\hat f_{2\sms n}f_1[h_{2n}+1]_q
$. Then
$
\bar f_{1\sms n}\hat f_{2\sms n}=\hat f_{2\sms n}\hat f_{1\sms n}
$.
\end{corollary}
\begin{proof}
The proof readily follows from  Lemma \ref{Dyn_h_Serre} and definition of $\hat f_{1\sms n}$ through
$\hat f_{2\sms n}$.
\end{proof}
\noindent
A straightforward refinement of Proposition \ref{key} extends to
$
e_1 \hat f_{1\sms n}=\hat f_{2\sms n}[h_{1\sms n}]_q+ \bar f_{1\sms n} e_1
$. Along with Corollary \ref{[f_2n,f_1n]}, this gives
\be
e_1\hat f_{1\sms n}^m =
[m]_q\hat f_{2\sms n} \hat f_{1\sms n}^{m-1} [h_{1\sms n}-m+1]_q+\bar f_{1\sms n}^m e_1.
\label{e_1f_1n^m}
\ee
Put formally $\hat f_{n+1\sms n}=1$. Corollary \ref{key} gives rise to the following result.
\begin{propn}
For an arbitrary weight $\la$ and a positive integer $m$,
$$
e_i\hat f_{k\sms n}^m v_\la=\dt_{ki}[m]_q [\la_{k\sms n}-m+1]_q \hat f_{k+1\sms n} \hat f_{k\sms n}^{m-1} v_\la,
$$
where $i,k=1,\ldots, n$.
\label{singular}
\end{propn}
\begin{proof}
The delta symbol is obvious. It is then sufficient to consider the case $i=k=1$.
This is an immediate consequence of  (\ref{e_1f_1n^m}).
\end{proof}
Corollary \ref{dyn_not0} with Proposition \ref{singular} gives
\begin{corollary}
The vector $\hat f_{k\sms n}^m v_\la$ is singular
if and only if $[\la_{k\sms n}-m+1]_q =0.$
\end{corollary}
For classical universal enveloping algebras, this result was obtained in \cite{Zh}.


\begin{thebibliography}{A}
\bibitem{Jan2} Jantzen, J. C.: Lectures on quantum groups.Grad. Stud. in Math., {\bf 6}. AMS, Providence, RI, 1996.
\bibitem{EK} Etingof, P., Kirillov, A.: Representations of affine Lie algebras, parabolic equations and Lam´e functions. Duke Math. J.  {\bf 74}, 585--614 (1994).
\bibitem{ES} Etingof, P.,  Schiffmann, O.:
Lectures on the dynamical Yang-Baxter equations, Quantum Groups and Lie Theory, London Math. Soc. LNS {\bf290 }(Durham, 1999) (2002).
\bibitem{AL} Alekseev, A., Lachowska, A.:
Invariant $*$-product on coadjoint orbits and the Shapovalov
pairing.
Comment. Math. Helv. {\bf 80}, 795--810  (2005).
\bibitem{KST}
 Karolinsky,   E., Stolin, A., Tarasov, V.:
Irreducible highest weight modules and equivariant quantization. Adv. Math. {\bf 211}, 266--283 (2007).

\bibitem{Sh} Shapovalov, N. N.: On a bilinear form on the universal enveloping algebra of a complex
semisimple Lie algebra, Funk. Anal. {\bf 6}, 65--70 (1972).
\bibitem{DCK} de Concini, C., Kac, V. G.: Representations of quantum groups at roots
of 1, Operator algebras, unitary representations, enveloping algebras,
and invariant theory (Paris, 1989), Progress in Mathematics, {\bf  92},
Birkh$\rm\ddot{a}$user, 1990, pp 471--506.
\bibitem{Jan1} Jantzen, J. C.:
Kontravariante formen und Induzierten Darstellungen halbeinfacher Lie-Algebren,
Math. Ann. {\bf 226},  53--65 (1977).


\bibitem{BHST} Burdk, C.,  Havlek, M., Smirnov, Yu.F., Tolstoy, V.N.: q-Analog of Gelfand-Graev Basis for the Noncompact
Quantum Algebra $U_q(u(n; 1))$, SIGMA 6 (2010), 010.
\bibitem{Mol} Molev, A.: Gelfand-Tsetlin bases for classical Lie algebras, in "Handbook of Algebra", Vol. 4, (M. Hazewinkel, Ed.), Elsevier, 2006, pp. 109--170.

\bibitem{ChP} Chari V. and Pressley A.:  A guide to quantum groups, Cambridge University Press, Cambridge, 1994.
\bibitem{D} Drinfeld,  V.:
   Quantum Groups. In Proc. Int. Congress of Mathematicians,
  Berkeley 1986, Gleason,  A. V.  (eds)  pp. 798--820, AMS, Providence (1987).
\bibitem{Mick} Mickelsson, J.: Step algebras of semi-simple subalgebras of Lie algebras, Reports
Math. Phys. {\bf 4}, 307--318 (1973).

\bibitem{AST} Asherova, R. M., Smirnov, Yu. F., and Tolstoy, V. N.: Projection operators for the
simple Lie groups, Theor. Math. Phys. {\bf 8}, 813-- 825 (1971).

\bibitem{KT} Khoroshkin, S. M., Tolstoy, V.N.: Extremal projector and universal R-matrix
for quantum contragredient Lie (super)algebras, in: Quantum Groups and Re-
lated Topics (R. Gielerak et al., eds.), Kluwer Academic Publishers, Dordrecht
1992, pp. 23--32.

\bibitem{O} Ostapenko, P.: Inverting the Shapovalov form. J. Algebra, {\bf 147}, 90--95 (1992).
\bibitem{ESt} Etingof, P.,  Styrkas, K.:
Algebraic Integrability of Macdonald Operators and Representations of Quantum Groups.
Comp. Math. {\bf 114}, 125--152 (1998).

\bibitem{Zh} Zhelobenko, D. P., An introduction to theory of S-algebras over S-algebras//
Representations of Lie groups and related topics, Adv. Study in Contemporary Maths. {\bf 7}, N.Y.: Gordon \& Breach 1990.

\end{thebibliography}
\end{document}